\documentclass[10pt]{amsart}
\makeatletter
 \def\@textbottom{\vskip \z@ \@plus 1pt}
 \let\@texttop\relax
\makeatother
\usepackage{amsxtra}
\usepackage[all]{xy}
\usepackage{palatino}
\usepackage{commath}
\usepackage{mathrsfs}
\usepackage{comment}

\usepackage{tikz-cd}
\usepackage[bookmarks=false]{hyperref}

\usepackage{amssymb,amsmath,amsthm,epsf,epsfig,dsfont,bbm}
 
\usepackage{upref, eucal}

\numberwithin{equation}{section}
\usepackage{amssymb,amsmath,amsthm,epsf,epsfig,dsfont,bbm}

\usepackage{latexsym}
\usepackage{mathtools}
\usepackage{xcolor}

\DeclareUnicodeCharacter{0301}{\'{e}}

\newcommand{\beqar}{\begin{eqnarray*}}
\newcommand{\eeqar}{\end{eqnarray*}}

\newcommand{\oldmarginpar}[1]{}

\begin{document}

\newcommand\J{\mathfrak{J}}
\newcommand\A{\mathbb{A}}
\newcommand\C{{\bf C}}
\newcommand\G{\mathbb{G}}
\newcommand\N{{\bf N}}
\newcommand{\T}{\mathbb{T}}
\newcommand{\cT}{\mathcal{T}}
\newcommand\sO{\mathcal{O}}
\newcommand\sR{\mathcal{R}}
\newcommand\sI{\mathcal{I}}
\newcommand\sE{{\mathcal{E}}}
\newcommand\tE{{\mathbb{E}}}
\newcommand\sF{{\mathcal{F}}}
\newcommand\sG{{\mathcal{G}}}
\newcommand\GL{{\mathrm{GL}}}
\newcommand{\HH}{\mathrm H}
\newcommand\mM{{\mathrm M}}
\newcommand\fS{\mathfrak{S}}
\newcommand\fP{\mathfrak{P}}
\newcommand\fQ{\mathfrak{Q}}
\newcommand\Qbar{{\bar{\Q}}}
\newcommand\sQ{{\mathcal{Q}}}
\newcommand\sP{{\mathbb{P}}}
\newcommand{\Q}{{\bf Q}}
\newcommand{\tH}{\mathbb{H}}
\newcommand{\Z}{{\bf Z}}
\newcommand{\R}{{\bf R}}
\newcommand{\F}{\mathbb{F}}
\newcommand{\D}{\mathfrak{D}}
\newcommand\gP{\mathfrak{P}}
\newcommand\Gal{{\mathrm {Gal}}}
\newcommand\SL{{\mathrm {SL}}}

\newcommand{\legendre}[2] {\left(\frac{#1}{#2}\right)}
\newcommand\iso{{\> \simeq \>}}
\newcommand{\cX}{\mathcal{X}}
\newcommand{\cJ}{\mathcal{J}}
\newtheorem{thm}{Theorem}
\newtheorem{theorem}[thm]{Theorem}
\newtheorem{cor}[thm]{Corollary}
\newtheorem{conj}[thm]{Conjecture}
\newtheorem{prop}[thm]{Proposition}
\newtheorem{lemma}[thm]{Lemma}
\theoremstyle{definition}
\newtheorem{definition}[thm]{Definition}
\theoremstyle{remark}
\newtheorem{remark}[thm]{Remark}
\newtheorem{example}[thm]{Example}
\newtheorem{claim}[thm]{Claim}
\newtheorem{lem}[thm]{Lemma}

\theoremstyle{definition}
\newtheorem{dfn}{Definition}

\theoremstyle{remark}

\theoremstyle{remark}
\newtheorem*{fact}{Fact}
\makeatletter
\def\imod#1{\allowbreak\mkern10mu({\operator@font mod}\,\,#1)}
\makeatother
\newcommand{\mcF}{\mathcal{F}}
\newcommand{\mcG}{\mathcal{G}}
\newcommand{\mcM}{\mathcal{M}}
\newcommand{\mcO}{\mathcal{O}}
\newcommand{\mcP}{\mathcal{P}}
\newcommand{\mcS}{\mathcal{S}}
\newcommand{\mcV}{\mathcal{V}}
\newcommand{\mcW}{\mathcal{W}}
\newcommand{\gS}{\mathcal{S}}
\newcommand{\cO}{\mathcal{O}}
\newcommand{\cC}{\mathcal{C}}
\newcommand{\mfa}{\mathfrak{a}}
\newcommand{\mfb}{\mathfrak{b}}
\newcommand{\mfH}{\mathfrak{H}}
\newcommand{\mfh}{\mathfrak{h}}
\newcommand{\mfm}{\mathfrak{m}}
\newcommand{\mfn}{\mathfrak{n}}
\newcommand{\mfp}{\mathfrak{p}}
\newcommand{\mfq}{\mathfrak{q}}
\newcommand{\cF}{\mathfrak{F}}
\newcommand{\mrB}{\mathrm{B}}
\newcommand{\mrG}{\mathrm{G}}
\newcommand{\gG}{\mathcal{G}}
\newcommand{\mrH}{\mathrm{H}}
\newcommand{\mH}{\mathrm{H}}
\newcommand{\mrZ}{\mathrm{Z}}
\newcommand{\Ga}{\Gamma}
\newcommand{\cyc}{\mathrm{cyc}}
\newcommand{\Fil}{\mathrm{Fil}}
\newcommand{\mrp}{\mathrm{p}}
\newcommand{\PGL}{\mathrm{PGL}}
\newcommand{\x}{{\mathcal{X}}}
\newcommand{\Sp}{\textrm{Sp}}
\newcommand{\ab}{\textrm{ab}}

\newcommand{\lra}{\longrightarrow}
\newcommand{\ra}{\rightarrow}
\newcommand{\rai}{\hookrightarrow}
\newcommand{\ras}{\twoheadrightarrow}

\newcommand{\repr}{\rho_{f,\wp}|_{G_p}}
\newcommand{\GRF}{{\rho}_{f,\wp}}

\newcommand{\lan}{\langle}
\newcommand{\ran}{\rangle}

\newcommand{\mo}[1]{|#1|}

\newcommand{\hw}[1]{#1+\frac{1}{2}}
\newcommand{\mcal}[1]{\mathcal{#1}}
\newcommand{\trm}[1]{\textrm{#1}}
\newcommand{\mrm}[1]{\mathrm{#1}}
\newcommand{\car}[1]{|#1|}
\newcommand{\pmat}[4]{ \begin{pmatrix} #1 & #2 \\ #3 & #4 \end{pmatrix}}
\newcommand{\bmat}[4]{ \begin{bmatrix} #1 & #2 \\ #3 & #4 \end{bmatrix}}
\newcommand{\pbmat}[4]{\left \{ \begin{pmatrix} #1 & #2 \\ #3 & #4 \end{pmatrix} \right \}}
\newcommand{\psmat}[4]{\bigl( \begin{smallmatrix} #1 & #2 \\ #3 & #4 \end{smallmatrix} \bigr)}
\newcommand{\bsmat}[4]{\bigl[ \begin{smallmatrix} #1 & #2 \\ #3 & #4 \end{smallmatrix} \bigr]}

\makeatletter
\def\imod#1{\allowbreak\mkern10mu({\operator@font mod}\,\,#1)}
\makeatother
\title{The Eisenstein cycles and Manin-Drinfeld properties}

\author{Debargha Banerjee}
\address{INDIAN INSTITUTE OF SCIENCE EDUCATION AND RESEARCH, PUNE, INDIA}
\author{Lo\"ic Merel}
\address{Universit\'e Paris Cit\'e and Sorbonne Universit\'e, CNRS, IMJ-PRG, F-75013 Paris, France}
\begin{abstract}
Let $\Ga$ be a subgroup of finite index of $\SL_2(\Z)$.
We give an analytic criterion for a cuspidal divisor to be torsion in the  Jacobian $J_\Ga$ of the corresponding modular curve $X_\Ga$. By Belyi's theorem, such a criterion would apply to any curve over a number field. Our main tool is the explicit description, in terms of modular symbols, of what we call {\em Eisenstein cycles}. The latter are representations of relative homology classes over which integration of any holomorphic differential forms vanishes.
Our approach relies in an essential way on the specific case $\Ga\subset\Ga(2)$, where we can consider convenient generalized  Jacobians instead of $J_\Ga$. 
The Eisenstein classes are the real part of certain homology classes with complex coefficients. The imaginary part of those classes are related to the scattering constants attached to Eisenstein series. 
Finally, we illustrate our theory by considering Fermat curves. 

 \end{abstract}

\maketitle

\section{Introduction}
\label{intro}
Let $\Ga$ be a subroup of finite index of $\SL_2(\Z)$, acting on the upper half-plane $\tH$. Consider the modular curve $Y_{\Gamma}=\Gamma \backslash \tH$.  Let $\partial_{\Ga}=\Ga\backslash \sP^1(\Q)$ be the set of cusps and $X_{\Gamma}$ be the compact modular curve obtained from $Y_{\Gamma}=\Gamma \backslash \tH$ by adding the cusps. 
We have a long exact sequence of relative homology: 
\[
0  \rightarrow  \HH_1(X_{\Ga};\Z) \rightarrow 
\HH_1(X_{\Gamma}, \partial_\Ga;\Z) \xrightarrow[]{\delta} \Z[ \partial_{\Ga}]
 \rightarrow  \Z \rightarrow 0,
\]
where the map $\HH_1(X_{\Gamma}, \partial_\Ga;\Z) \xrightarrow[]{\delta} \Z[ \partial_{\Ga}]$ is obtained from the boundary map $\HH_1(X_{\Gamma}, \partial_\Ga;\Z) \rightarrow \HH_0( \partial_\Ga;\Z)$.
Let
 $\HH^0(X_{\Gamma}, \Omega^1)$ 
denote the complex vector space of global sections of the sheaf of holomorphic differentials on the Riemann surface $X_{\Gamma}$. 
Let $D=\sum_{j\in \partial_\Ga}m_j[j]$ be a divisor of degree $0$ supported on $\partial_{\Ga}$.  
We call {\it Eisenstein classes} corresponding to $D$ to be the unique element $\sE_D\in \HH_1(X_{\Gamma}, \partial_{\Ga}; \R)$ such that $\int_{\sE_D}\omega=0$ for all $\omega\in \HH^0(X_{\Gamma}, 
\Omega^1)$ and $\delta(\sE_D)=D$. We will determine this class in this article, with the following application in mind: the divisor $D$ is torsion in the  Jacobian $J_\Gamma$ of $X_\Ga$ if and only if $\sE_D\in \HH_1(X_{\Gamma}, \partial_{\Ga}; \Q)$. 

It is known since Manin and Drinfeld that $D$ is torsion when $\Ga$ is a congruence subgroup \cite{MR0318157,MR0314846}. Scholl ~\cite{MR809492}, Murty-Ramakrishnan~\cite{MR983619} and recently Burrin ~\cite{Burrin} have given criteria for $D$ being torsion in $J_\Gamma$, without appealing to $\sE_D$.
Our approach to this question is different from those authors and is a continuation of \cite{DebarghaMerel} and \cite{MR1405312} (where we limited ourselves to congruence subgroups).  

We impose a more specific setting (as we will see, without losing generality), in connection with the theory of Belyi maps and dessins d'enfants, see section \ref{Belyi}. We assume in this article that $\Ga$ is contained in $\Ga(2)$ (the principal congruence subgroup of level $2$) and that $-{\rm Id}\in\Ga$. The corresponding modular curve $X(2)=X_{\Ga(2)}$ has three cusps corresponding to the classes of $0$, $1$ and $\infty$. Denote by $\partial_{\Ga}^{+}$ (resp. $\partial_{\Ga}^{-}$) the cusps above either $0$ or $\infty$ (resp. above $1$) via the canonical morphism $X_{\Ga}\rightarrow X(2)$. 

We reformulate the question of the Eisenstein cycles in the mixed homology group $\HH_1(X_{\Ga}-\partial_{\Ga}^{-}, \partial_{\Ga}^{+};{\bf R})$ and its dual $\HH_1(X_{\Ga}-\partial_{\Ga}^{+}, \partial_{\Ga}^{-};{\bf R})$. 
Let $\alpha$, $\beta$ be rational numbers with odd numerators and denominators (resp. with either even numerator or even denominator). Denote by ${\{\alpha,\beta\}}^-$ (resp. $\{\alpha,\beta\}^{+}$) the class in $\HH_1(X_{\Ga}-\partial_{\Ga}^{+}, \partial_{\Ga}^{-}; {\bf R})$ (resp. $\HH_1(X_{\Ga}-\partial_{\Ga}^{-}, \partial_{\Ga}^{+}; {\bf R})$) of the geodesic path in $\tH$ from $\alpha$ to $\beta$. 
Such a consideration is justified by the fact that the elements $\{g0,g\infty\}^{+}$ and $\{g(1),g(-1)\}^{-}$ form bases of the free $\Z$-modules $\HH_1(X_{\Ga}-\partial_{\Ga}^{-}, \partial_{\Ga}^{+};  \Z)$ and $\HH_1(X_{\Ga}-\partial_{\Ga}^{+}, \partial_{\Ga}^{-}; \Z)$ 
respectively when $g$ runs though a set of representatives of $\Ga\backslash\Ga(2)$ \cite{MR1405312}, and see section \ref{MC}.

For $j\in \partial_{\Ga}$, let $\Ga_{j}$ the stabilizer of $\sigma_{j}(\infty)$ in $\Ga$. For $x\in\sP^1(\Q)$, such that $m_{\Ga x}=0$, set 
\[
\gS_D(x)=\frac{1}{2\pi i}\lim_{s\rightarrow 1^{+}}\sum_{j\in\partial_\Ga}m_j\sum_{r=1}^{\infty}\sum_{\sigma \in \Ga_j \backslash \Ga \slash \Ga_j} \frac{e^{2\pi ir(x+\frac{d}{c})}}{|c|^{2s}},
\]
where the entries $c$ and $d$ are such that  $\sigma_j^{-1} \sigma \sigma_j=\left(\begin{smallmatrix} \star & \star\\
c & d\\
\end{smallmatrix}\right)$ and $c$ is required to be $>0$.
The internal sum over $\sigma$ is sometimes called a Kloosterman zeta function  \cite[Chapter 9, p. 121]{MR1942691}, \cite[p. 14]{MR809492}. {\it A priori}, such Dirichlet series converge absolutely for $\Re(s)>1$. Their  values at $1$ exist since $D$ is of degree $0$.

Let $D^{+}\in\text{Div}^0(\partial_{\Ga}^{+})$ (resp. $D^{-}\in\text{Div}^0(\partial_{\Ga}^{-})$). Define $F_{D^{+}}$ and  $F_{D^{+}}$: $ \Ga(2) \rightarrow \C$ by 
\[
F_{D^{+}}(h) =
\gS_{D^{+}}(h(-1))-\gS_{D^{+}}(h(1))
\hskip1truecm{\rm and }\hskip1truecm
F_{D^{-}}(h) =\gS_{D^{-}}(h(\infty))-\gS_{D^{-}} (h(0)).
\]
We will see that those $\Ga$-invariant quantities can be understood as periods of certain Eisenstein series.
Set 
$$
\sE_{D^{+}}= \sum_{g \in \Ga \backslash \Ga(2)} F_{D^{+}}(g) \{g0,g\infty\}^{+}
\hskip1truecm{\rm and }\hskip1truecm
\sE_{D^{-}}= \sum_{g \in \Ga \backslash \Ga(2)} F_{D^{-}}(g) \{g(1),g(-1)\}^{-}
$$
in $\HH_1(X_{\Ga}-\partial_{\Ga}^{-}, \partial_{\Ga}^{+};  {\C})$ and $\HH_1(X_{\Ga}-\partial_{\Ga}^{+}, \partial_{\Ga}^{-}; {\C})$ respectively.
We decompose them into their real and imaginary parts:
$\sE_{D^{+}}=\sR_{D^{+}}+\frac{1}{2\pi i}\sI_{D^{+}}$
and 
$\sE_{D^{-}}=\sR_{D^{-}}+\frac{1}{2\pi i}\sI_{D^{-}}$.

\begin{thm}
\label{Eisensteincycles}
For all differential forms of the third kind $\omega$ on $X_{\Ga}$ with poles in $\partial_{\Ga}^{-}$ (resp. $\partial_{\Ga}^{+}$), one has $\int_{\sR_{D^{+}}}\omega=0$ (resp. $\int_{\sR_{D^{-}}}\omega=0$) and the boundary of $\sR_{D^{+}}$ (resp. $\sR_{D^{-}}$) is $-D^{+}$ (resp.  $-D^{-}$). Moreover $\sE_{D^{+}}$ (resp. $\sE_{D^{-}}$) has boundary $D^{+}$ (resp. $D^{-}$).
\end{thm} 

It follows that the images of $\sR_{D^{+}}$ and $\sR_{D^{-}}$ in $\HH_1(X_{\Gamma}, \partial_{\Ga}; {\bf R})$ are Eisenstein classes. They are the classes of what we call {\it Eisenstein cycles}. We will explain in section~\ref{GeneralEisensteinclasses} how to derive all Eisenstein classes in $\HH_1(X_{\Gamma}, \partial_{\Ga}; {\bf R})$  from such images. We show in section \ref{scattering} that $\sI_{D^{-}}$ and $\sI_{D^{+}}$ are deduced from the scattering constants, obtained from specializing scattering matrices of Eisenstein series at $s=1$.

We have a canonical map $\lambda^+$ : $\HH_{0}(\partial_{\Ga}^{-},\R)\rightarrow \HH_1(X_{\Ga}-\partial_{\Ga}^{-}, \partial_{\Ga}^{+},  \R)$  (resp. 
$\lambda^-$ : $\HH_{0}(\partial_{\Ga}^{+},\R)\rightarrow \HH_1(X_{\Ga}-\partial_{\Ga}^{+}, \partial_{\Ga}^{-},  \R)$) dual to the boundary map. Its image contains $\sI_{D^{+}}$ (resp. $\sI_{D^{-}}$), see section \ref{proof}.
Let $J_\Gamma^-$ (resp. $J_\Gamma^+$) be the generalized  Jacobian of $X_\Gamma$ with respect to the set $\partial_{\Ga}^{-}$ (resp. $\partial_{\Ga}^{+}$).
\begin{theorem}
 \label{Manin-Drinfeld}
 The class of $D^{+}$ (resp. $D^{-}$) in $J_\Gamma^-$ (resp. $J_\Gamma^+$) is torsion if and only if for every $g \in \Ga \backslash \Ga(2)$, one has $F_{D^{+}}(g)\in \Q$ (resp. $F_{D^-}(g)\in \Q$).
 
The class of  $D^{+}$ (resp. $D^{-}$) in $J_\Gamma$ is torsion if and only if one has $\sum_{g \in \Ga \backslash \Ga(2)} F_{D^{+}}(g)[g]\in \Q[ \Ga \backslash \Ga(2)]+\lambda^+(\R[\partial_\Gamma^-])$ (resp. $\sum_{g \in \Ga \backslash \Ga(2)} F_{D^{-}}(g)[g]\in \Q[ \Ga \backslash \Ga(2)]+\lambda^-(\R[\partial_\Gamma^+])$.
 
 \end{theorem}
The first criterion is similar to the following statement: $e^{2\pi iz}$ is torsion in $\C^{\times}$ if and only if the real part of $z$ is rational, and the imaginary part of $z$ is $0$.
Note that the first statement is substantial even if $X_\Ga$ has genus $0$.  
The second criterion can be derived from the first. The maps $\lambda^+$ and $\lambda^-$ are computed explicitly by proposition \ref{boundary}.

The principal weakness of our approach resides in the fact that the functions $F_{D^{+}}$ and $F_{D^{-}}$ are not easily computable for two reasons : {\it (i)} the convergence of Dirichlet series is notoriously difficult to understand, and {\it (ii)} we have no convenient way to pass from a combinatorial description of $P\Ga(2)=\Ga(2)/\{\pm 1\}$ as a subgroup of the free group on two generators $P\Ga(2)$ to a description of the elements of $\Ga$ as $2\times2$ matrices. But it is expected that the Eisenstein classes are of a transcendental nature in general, and we hope that our method cast a new light on this nature.

We computed the functions $F_D^{+}$ and $F_D^{-}$ when $\Ga$ is congruence subgroup  \cite{MR3251709,Pacific,DebarghaMerel}. Here is  a classical instance of a non-congruence subgroup where the computation of these formulas is possible.
 
Let $N$ be an integer $>0$. Let $\Phi_N$ be the kernel of the composed maps $\Ga(2)\rightarrow\Ga(2)^{\rm ab}\rightarrow\Ga(2)^{\rm ab}/N$, where the first map is the abelianization. It is a non-congruence subgroup, when $N>8$. The corresponding modular curve $X_{\Phi_N}$ is the $N$-th Fermat curve ({\it i.e.} it admits the model $X^N+Y^N=1$, the cusps correspond to the points such that $X^N=0$, $1$ or $\infty$). 
Rohrlich has shown that every cuspidal divisor of degree $0$ is torsion in $J_{\Phi_N}$ \cite{MR0441978}. As an application of our theory, we describe the cuspidal subgroups of $J_{\Phi_N}^+$ and $J_{\Phi_N}^-$; as expected they are finite.

 \subsection{Acknowledgements}
It is a pleasure to acknowledge several e-mail communication, advice and remark of Professors Rohrlich and  Professor Kumar Murty. 
The first author was partially supported by the SERB grant MTR$/2017/000357$ and CRG$/2020/000223$.

\section{The combinatorics of Belyi maps}
 \label{Belyi}
 
 Let $X$ be a smooth projective curve over ${\C}$. 
 A {\it Belyi} map is a morphism $f:  X\rightarrow\sP^1(\C)$ unramified outside $0$, $1$, $\infty$  \cite{MR1757192}.
 
 A {\it dessin d'enfant} \cite{MR1483107} is a connected graph, composed of a finite set of vertices  $V$ and a finite set of edges $E$, with the additional structures :
 \begin{itemize}
  \item a map from $V$ to the pair $\{0,\infty\}$ (the graph is bicolored)
  \item the extremities of a given edge have different colors
  \item the set of of edges adjacent to a given vertice is endowed with a cyclic ordering ({\it i.e.} a transitive action of $\Z$).
\end{itemize}
 
The dessin d'enfant associated to $f$ is the bicolored graph $G$ whose vertices of color $0$ (resp. $\infty$) constitute $f^{-1}(0)$ (resp. $f^{-1}(\infty)$).
The edges of the graph are the components the inverse images by $f$ of the arc form $\infty$ to $0$ in $\sP^1(\C)$ (we depart here from the usual convention, where $0$ and $1$ are used instead of $0$ and $\infty$). 

Consider the function on the upper half-plane  \cite[p. 86 and p. 105]{MR808396}:
$$
\lambda(z)=16 e^{2\pi iz} \prod_{n \geq 1} \left( \frac{1+e^{2n\pi iz}}{1+e^{(2n-1)\pi iz}}\right)^8.
$$
It identifies the modular curve $X(2)$ to the projective line. We resist the tentation to use $1/\lambda$; instead we follow here the customary convention so that $\lambda(\infty)=0$, $\lambda(0)=\infty$, $\lambda(1)=\lambda(-1)=1$.
Thus the arc from $\infty$ to $0$ on the projective line is the image by $\lambda$ of the arc from $0$ to $\infty$ in the upper half-plane.

Since $\Gamma(2)/\{-1,1\}$ identifes to the fundamental group of the affine modular curve $Y(2)$, and the upper half-plane is the universal covering of $Y(2)$, the morphism $f$ can be regarded as the morphism $X_{\Ga}\rightarrow X(2)$, for an appropriate subgroup $\Ga\subset\Ga(2)$ of finite index, and containing $-{\rm Id}$. The group $\Gamma(2)/\{-1,1\}$ is freely generated by $\{A,B\}$ where $A$ is the image of $\left(\begin{smallmatrix}
1 & 2\\
 0 & 1\\
\end{smallmatrix}\right)$ and $B$ is the image of $\left(\begin{smallmatrix}
1 & 0\\
2 & 1\\
\end{smallmatrix}\right)$. 

The dessin d'enfant attached to $X_\Ga$ is the graph whose vertices are the cusps of $X_\Gamma$ above $0$ (of colour $\infty$) and $\infty$ (of colour $0$). The edges are the translates by $\Ga(2)$ of the image in $X_\Ga$ of the arc from $0$ to $\infty$ (in the upper half-plane). The action of $\Z$ on the edges attached to a vertice of color $0$ (resp. $\infty$) follows from the action of $A$ (resp. $B^{-1}$). 

Conversely, $\Gamma$ can be recovered as follows from the dessin. It acts on the edges of the dessin: consider an edge $e$ attached to a pair of vertices $v_e(0)$ and $v_e(\infty)$ (coloured $0$ and $\infty$ respectively), the image of $e$ by $A$ (resp. $B$) is the successor (resp. predecessor) of $e$ along the vertice $v_e(0)$ (resp. $v_e(\infty)$). The stabilizers of all edges constitute $\Gamma$. Thus a dessin d'enfant encodes the situation considered in this article.

{\it Question}: Is there a natural interpretation of the Kloosterman zeta function purely in terms of the geometry of the dessin, without involving the entries of matrices in $\Ga(2)$?

\section{Preliminaries and previous work}
\label{Riemann}
Let $X$ be a compact, connected, non-empty Riemann surface. Let $S$ be a finite subset of $X$. Let $D$ be a divisor of degree $0$ supported on $S$. Recall that a differential of the third kind on $X$ is a meromorphic differential whose poles are simple and whose residues are integers. 

Let $\omega$ be a differential of the third kind on $X$, with divisor of poles equal to $D$. It exists by the Riemann Roch theorem.
There exists a unique holomorphic differential form $\omega'\in \Omega^{1}(X)$ such that, for every $c\in {\HH}_{1}(X, S; \Z)$, one has ${\rm Re}(\int_{c}\omega)={\rm Re}(\int_{c}\omega')$. 
Therefore there exists a unique differential of the third kind $\omega_{D}$ of residue divisor $D$ such that ${\rm Re}(\int_{c}\omega_{D})=0$ for every  $c\in {\HH}_{1}(X, S; \Z)$.
This is the {\it canonical differential of the third kind} associated to $D$. The notion extends obviously to the situation where $D$ is a divisor of degree $0$ with real coefficients.

Let $f_D$ be a multivalued function on $X-S$ such that $df_D=\omega_{D}$. Since the periods of $\omega_{D}$ are imaginary, the real part $g_{D}$ of $f_D$ is single-valued and harmonic.

We turn now to the specific situation of modular curves. Let $\Gamma$ be a subgroup of finite index of the modular group.  
The cases of noncongruence subgroups has been studied since Atkin and Swinnerton-Dyer~\cite{MR0337781}.  
Suppose $X$ is the modular curve $X_{\Gamma}$ associated to $\Gamma$, and $S$ is the set $ \partial_{\Ga}$ of cusps.
The pullbacks to the upper half-plane of differential of the third kind care of the form $2\pi if(z)\,dz$, where $f$ is a holomorphic modular form of weight $2$ for $\Gamma$.
The pullback of the {\it canonical differential of the third kind} associated to $D$ is of the form $2\pi iG_{D}(z)\,dz$, where $G_{D}$ is by definition the {\it Eisenstein series} associated to $D$.

The Eisenstein series $G_{D}$ has been determined explicitly see section  \ref{EC}. We find convenient to follow Scholl's account, which in turn follows Hecke, Selberg, Kubota etc.

Let $j\in \partial_{\Ga}$ and choose $\sigma_{j}\in \SL_2(\Z)$ such that $\sigma_{j}\infty\in j$ and denote by $\Ga_{j}$ the stabilizer of $\sigma_{j}\infty$ in $\Ga$. 
For $s\in \C$, with ${\rm Re}(s)>1$ and $z=x+iy\in \tH$, set
\[
E_{j}(z,s)=\sum_{\sigma\in \Ga_{j}\backslash \Ga}{\rm Im}(\sigma_{j}^{-1}\sigma(z))^{s}=\sum_{\sigma\in \Ga_{j} \backslash \Ga}\frac{y^{s}}{|cz+d|^{s}},
\]
where $(c,d)$ is the lower row of $\sigma_{j}^{-1}\sigma$. As a function of $s$, $E_{j}(z,s)$ admits a meromorphic continuation to $\C$, with a simple pole  at $s=1$.

Its study involves the corresponding Kloosterman zeta function at the pair of cusps $(j,k)$, which has been introduced by Selberg, even though the terminology seems due in print to Goldfeld and Sarnak ~\cite{MR689644}, see also the monographies of Kubota \cite[Chapter 2, p. 16]{MR0429749} and Iwaniec \cite[Chapter 9, p. 121]{MR1942691}. For $r\in\Z$, it is defined as
\[
\phi_{jk,r}(s):=\sum_{\sigma \in \Ga_j \backslash \Ga \slash \Ga_k} \frac{e^{\frac{2 \pi i r d}{c}}}{|c|^{2s}},
\]
where the entries $c$ and $d$ are such that  $\sigma_j^{-1} \sigma \sigma_k=\left(\begin{smallmatrix}
\star & \star\\
c & d\\
\end{smallmatrix}\right)$ and $c$ is required to be $>0$.
The series is absolutely convergent if ${\rm Re}(s)>1$, and can be extended to a meromorphic function on the complex plane, which is due to Selberg ~\cite{MR0182610}, see for instance ~\cite[Chapter 9, p. 122]{MR1942691}. 
It has a pole at $s=1$ when $r=0$, and extends holomorphically to the neighbourhood of $s=1$, when $r\ne0$.
The notation  $\phi_{jk, r}$ comes from Kubota's book \cite{MR0429749}, and has been retained by Goldstein and Scholl \cite{MR0318065,MR809492}, and has been modified by Iwaniec in \cite{MR1942691}.
The expansion of $E_{j}$ at $\sigma_{k}(\infty)$ is given by the formula :
\begin{equation}
\label{development}
E_{j}(\sigma_{k}(z),s)=\sum_{r\in {\bf Z}} a_{jk,r}(y,s)e^{2\pi i rx}.
\end{equation}
where, for $r\ge1$, 
\begin{equation}
a_{jk,r}(y,s)=\phi_{jk,r}(s)F_{r}(y,s),
\end{equation}
and
\begin{equation}
\label{constantterm}
a_{jk,0}(y,s)=\delta_{j,k}y^{s}+\phi_{jk,0}(s)y^{1-s}\pi^{1/2}\frac{\Gamma(s-1/2)}{\Gamma(s)}
\end{equation}
here $\delta_{j,k}$ is the Kronecker symbol, and $F_{r}(y,s)$ is built out of Bessel functions, and does not concern us, except for the value at $s=1$. 
Note that Scholl's formula for $a_{jk,0}(y,s)$ seems to contain a misprint: it involves $\pi^{s}$ instead of $\pi^{1/2}$, contra Kubota and Iwaniec.
The coefficients $\phi_{jk,0}(s)\pi^{1/2}{\Gamma(s-1/2)}/{\Gamma(s)}$ are the entries of the {\it scattering matrix} at the pair $(j,k)$. 
Such expressions have all a simple pole at $s=1$, whose residue does not depend on $j$ and $k$. More precisely, the following limit exists
\begin{equation}
\label{scattering development}
C_{j,k}=\lim_{s\to1}(\phi_{jk,0}(s)\pi^{1/2}\frac{\Gamma(s-1/2)}{\Gamma(s)}-\frac{2}{\pi|\Ga(2)/\Ga|(s-1)})=\lim_{s\to1}(\pi\phi_{jk,0}(s)-\frac{2}{\pi|\Ga(2)/\Ga|(s-1)}).
\end{equation}
It is the {\it scattering constant}, see \cite{MR2140212}, at the pair $(i,j)$. As a remark on the terminology, we note that the term seems to be of more recent use than the calculations that justifies its existence. It does not appear in \cite{MR1942691}, for instance.
The constants could be, and have been \cite{Posingies}, normalized differently, for instance par using $\zeta(s)$ instead of $1/(s-1)$. Since we will only be interested in differences between those constants, this doesn't matter to us.

Following Scholl~\cite[p. 15]{MR809492}, define 
\[
G_j(z,s)= 2 i \frac{\partial}{\partial z} E_j(z,s).
\]

Furthermore, write $G_j(z)=\lim_{s \rightarrow 1}G_j(z,s)$.
The Kloosterman zeta function is involved in the $q$-expansion of $G_D$ by the formula \cite{MR809492} (based on \cite{MR0429749})
\[
G_j(\sigma_k(z))J(\sigma_j,z)^{-2}=\delta_{j,k}-\frac{\pi C}{y}-4\pi^2\sum_{r=1}^\infty r\phi_{jk,r}(1)q^r
\]
where $z=x+iy$, $q=e^{2\pi iz}$, $J(\gamma,z)=cz+d$ when $\gamma=\left(\begin{smallmatrix}
\star & \star\\
c & d\\
\end{smallmatrix}\right)$, $C$ is a constant independent of $j$, and $\delta_{j,k}$ is the Kronecker symbol.
Note that if we write $G_{j|\sigma_j}(z)=G_j(\sigma_j(z))J(\sigma_j,z)^{-2}$, one has $G_{j|\sigma_j}=G_{\sigma_j^{-1}\Ga j}=G_{\sigma_j^{-1}\Ga\sigma_j \infty}$.

Write $D=\sum_{j\in\partial_{\Ga}}m_{j}[j]$. The limit $\lim_{s\to 1}E_j(z,s)$ makes no sense, because of the pole at $s=1$. But  $E_D(z)=\lim_{s\to 1}\sum_{j\in\partial_\Ga}m_jE_j(z,s)$ exists, as the residue at $s=1$ disappears. 
The modular form $G_D$ is given by 
\[
G_D(z)=\sum_{j=1}^m m_j G_j(z).
\]
It is holomorphic since $D$ is of degree $0$. Write
\[
G_D(z)=\sum_{r=0}^\infty a_re^{2\pi irz}.
\]
Scholl, using Waldschmidt's work in transcendental number theory ~\cite{MR0570648}, established the following criterion.

\begin{thm}[Scholl,~\cite{MR809492}]
\label{Scholl}
The divisor $D$ is torsion in $J_\Gamma$ if and only if, for every integer $r\ge1$, the coefficient 
\begin{equation}
\label{coefficient}
a_r=-4\pi^2r\sum_{j\in\partial_\Ga}m_j\phi_{j\Ga \infty,r}(1)
\end{equation}
is an algebraic number. 
\end{thm} 

Note that the differential form $G_D(z)\, dz$ is determined by finitely many (in terms of the genus of $X_\Ga$) coefficients $a_r$. Therefore the criterion can be reduced to verifying the algebricity of finitely many numbers.
The coefficients $a_{r}$ have been computed slightly more explicitly by K. Murty and Ramakrishnan, in terms of what they call generalized Ramanujan sums \cite{MR983619}.

In  \cite{MR2140212}, K\"uhn expressed the N\'eron-Tate pairing of two divisors of degree zero supported on $\partial_{\Ga}$ as a rational linear combination of (1) logarithms of integers and (2) products of $\pi$ by scattering constants, at least when $X_\Ga$ is defined over ${\Q}$. He derived as a consequence a formula for the scattering constants when all such divisors are torsion in $J_{\Ga}$.

\section{Mixed homology groups}
\label{MC}
We retain the notations of the introduction.
We follow \cite{MR1405312} and study the mixed homology groups $\HH_1(X_\Gamma-\partial_{\Ga}^{-}, \partial_{\Ga}^{+};\Z)$ and 
$\HH_1(X_\Gamma-\partial_{\Ga}^{+}, \partial_{\Ga}^{-};\Z)$. The intersection pairing provides a perfect bilinear pairing 
\[
\bullet:\HH_1(X_\Gamma-\partial_{\Ga}^{+},\partial_{\Ga}^{-};\Z) \times \HH_1(X_\Gamma-\partial_{\Ga}^{-},\partial_{\Ga}^{+};\Z) \rightarrow \Z.
\]
For $g \in \Gamma(2)$, set $\xi^{+}(\Ga g)=\{g0,g\infty\}^{+}$ and $\xi^{-}(\Ga g)=\{g1,g(-1)\}^{-}$.
By linearity, we extend the maps $\xi^{+}$ and $\xi^{-}$ to $\Z[\Gamma \backslash \Gamma(2)]\rightarrow \HH_1(X_\Gamma-\partial_{\Ga}^{-},\partial_{\Ga}^{+};\Z)$ and  $\Z[\Gamma \backslash \Gamma(2)]\rightarrow \HH_1(X_\Gamma-\partial_{\Ga}^{+},\partial_{\Ga}^{-};\Z)$ respectively.

\begin{thm}[~\cite{MR1405312}]
\label{int-iso}

The map $\xi^{+}$ and $\xi^{-}$ thus obtained are group isomorphisms. Furthermore, for $g$, $h\in \Ga(2)$ the intersection pairing $\xi^{+}( g)\bullet\xi^{-}(h)$ is equal to $1$ if $\Gamma g=\Gamma h$ and to $0$ otherwise.

\end{thm}

Consider the maps: $\HH_1(X_\Gamma-\partial_{\Ga}^{-}, \partial_{\Ga}^{+};\Z)\rightarrow \HH_1(X_\Gamma, \partial_{\Ga}^{+};\Z)$ and $\HH_1(X_\Gamma-\partial_{\Ga}^{+}, \partial_{\Ga}^{-};\Z)\rightarrow \HH_1(X_\Gamma, \partial_{\Ga}^{-};\Z)$. 
The kernel of those maps are the image of the maps $\lambda^{+}$: $\Z[\partial_{\Ga}^{-}]\simeq\HH_{0}(\partial_{\Ga}^{-},\Z)\rightarrow\HH_1(X_\Gamma-\partial_{\Ga}^{+}, \partial_{\Ga}^{-};\Z)$ and 
$\lambda^{-}$: $\Z[\partial_{\Ga}^{+}]\simeq\HH_{0}(\partial_{\Ga}^{+};\Z)\rightarrow\HH_1(X_\Gamma-\partial_{\Ga}^{-}, \partial_{\Ga}^{+};\Z)$ respectively. To be precise, $\lambda^{-}(x)$ (resp. $\lambda^{+}(x)$) is homologous to a counterclockwise loop around the cusp $x$. The width of the cusp is the ramification index of $\pi_{0}$ at that cusp. Recall that $A$ (resp. $B^{-1}$, resp. $BA^{-1}$) is the generator of the stabiliser of 
$\infty$ (resp. $0$, resp. $1$) in $P\Ga(2)$ such that, for $z$ in the upper half-plane, the image in $Y(2)$ of a path from $z$ to $Az$ (resp. $B^{-1}z$, resp. $BA^{-1}z$) is a counterclockwise loop around $\Gamma(2) \infty$ (resp. $\Ga(2)0$, resp. $\Ga(2)1$).

\begin{prop}
\label{boundary}
Let $g\in \Ga(2)$ and denote by $w_{\infty}$ the width of the cusp $j=\Gamma g\infty$. One has 
$$
\lambda^{-}(\Gamma g\infty)=-\sum_{k=0}^{w_{\infty}-1}\xi^{-}(gA^{k})=\sum_{h\in \Ga\backslash\Ga(2),h\infty=j}\xi^-(h).
$$
Denote by $w_{0}$ the width of the cusp $\Gamma g0$. One has 
$$
\lambda^{-}(\Gamma g0)=\sum_{k=0}^{w_{0}-1}\xi^{-}(gB^{k})=\sum_{h\in \Ga\backslash\Ga(2),h0=j}\xi^-(h).
$$
Denote by $w_{1}$ the width of the cusp $\Gamma g1$. One has 
$$
\lambda^{+}(\Gamma g1)=\sum_{k=0}^{w_{1}-1}\xi^{+}(g(AB^{-1})^{k}B)-\xi^{+}( g(AB^{-1})^{k})=\sum_{h\in \Ga\backslash\Ga(2),h(-1)=j}\xi^-(h)-\sum_{h\in \Ga\backslash\Ga(2),h1=j}\xi^-(h).
$$

\end{prop}

\begin{proof}
Note that $A\infty =\infty$, $B0=0$, $A(-1)=B(-1)=1$. The class of a loop around the cusp $\Gamma g\infty$ is given by 
\begin{eqnarray*}
\{g(-1), (gAg^{-1})^{w_{\infty}}g(-1)\}^{-}
&= &
\{g(-1),gA^{w_{\infty}}(-1)\}^{-}\\
&= &
\{g(-1),gA^{w_{\infty}-1}1\}^{-}\\
&= &
\{g(-1),gA^{w_{\infty}-1}(-1)\}^{-}+\{gA^{w_{\infty}-1}(-1),gA^{w_{\infty}-1}1\}^{-}\\
&= &
\{g(-1),gA^{w_{\infty}-1}(-1)\}^{-}+\{gA^{w_{\infty}-1}(-1),gA^{w_{\infty}-1}1\}^{-}\\
&= &
\{g(-1),gA^{w_{\infty}-1}(-1)\}^{-}-\xi^{-}(gA^{w_{\infty}-1}).\\
\end{eqnarray*}
Which gives the first formula by iteration. The second formula is proved by the same method (replace $A$ by $B^{-1}$).

The third formula is obtained similarly. The class of a loop around the cusp $\Ga g1$ is given by 
\begin{eqnarray*}
\{g\infty,(gBA^{-1}g^{-1})^{w_{1}}g\infty\}^{+}
&= &
\{g\infty,g(BA^{-1})^{w_{1}-1}B\infty\}^{+}\\
&= &
\{g\infty,g(BA^{-1})^{w_{1}-1}B0\}^{+}+\{g(BA^{-1})^{w_{1}-1}B0, g(BA^{-1})^{w_{1}-1}B\infty\}^{+}\\
&= & 
\{g\infty,g(BA^{-1})^{w_{1}-1}0\}^{+}+\xi^{+}(g(BA^{-1})^{w_{1}-1}B)\\
&= &
\{g\infty,g(BA^{-1})^{w_{1}-1}\infty\}^{+}+\{g(BA^{-1})^{w_{1}-1}\infty,g(BA^{-1})^{w_{1}-1}0\}^{+}\\
&&+\xi^{+}(g(AB^{-1})^{w_{1}-1}B)\\
&= &
\{g\infty,g(BA^{-1})^{w_{1}-1}\infty\}^{+}-\xi^{+}(g(BA^{-1})^{w_{1}-1})+\xi^{+}(g(AB^{-1})^{w_{1}-1}B),\\
\end{eqnarray*}

which leads to the third formula, by iterating again.

\end{proof}

The boundary maps $\HH_1(X_\Gamma-\partial_{\Ga}^{-},\partial_{\Ga}^{+};\Z) \rightarrow \Z[\partial_{\Ga}^{+}]$ (resp. $\HH_1(X_\Gamma-\partial_{\Ga}^{+},\partial_{\Ga}^{-};\Z) \rightarrow \Z[\partial_{\Ga}^{-}]$)  associates to $\xi^{-}(g)$ (resp. $\xi^{+}(g)$) the divisor $(\Ga g1)-(\Ga g(-1))$ (resp. $(\Ga g\infty)-(\Ga g0)$). They are dual to the maps $\lambda^{-}$ and $\lambda^{+}$ respectively.

There is a notion of Eisenstein class in $\HH_1(X_\Gamma-\partial_{\Ga}^{-},\partial_{\Ga}^{+};\R)$ and $\HH_1(X_\Gamma-\partial_{\Ga}^{+},\partial_{\Ga}^{-};\R)$. 
Let $D^+$ (resp. $D^-$) be a divisor of degree $0$ supported on $\partial_\Ga^+$ (resp. $\partial_\Ga^-$). 
The corresponding Eisenstein class belongs to $\HH_1(X_\Gamma-\partial_{\Ga}^{-},\partial_{\Ga}^{+};\R)$ (resp. $\HH_1(X_\Gamma-\partial_{\Ga}^{+},\partial_{\Ga}^{-};\R)$).

It is the unique element $c$ of boundary $D^+$ (resp. $D^-$) such that $\Re(\int_c\omega)=0$ for all $\omega$ differential form of the third kind whose poles are supported on $\partial_\Ga^-$ (resp. $\partial_\Ga^+$) and whose residues are real.

\section{Inner product formula}
\label{innerprod}
Let $M_{2}(\Gamma)$ be the space of holomorphic modular forms of weight $2$ for $\Ga$. 
The following formula is akin to the formula \cite[p. 21, Theorem 2]{MR2641193}. 
\begin{theorem}
\label{pinnerprodimp}
Let $f_{+}$ and $f_{-}$ be elements $M_{2}(\Gamma)$ such that the poles of $f_{+}(z)dz$ and $f_{-}(z)dz$ belong to $\partial_{\Ga}^{+}$ and $\partial_{\Ga}^{-}$ respectively. 
Thus the Petersson inner product $<f_{+}, f_{-}>$ is well defined. We have the formula:

\[
<f_{+}, f_{-}>=\frac{1}{12i[\Ga(2): \Ga]} \sum_{g \in \Ga \backslash \Ga(2)} \int_{g(1)}^{g (-1)} f_{+}(z)\,dz \int_{g 0}^{g \infty}\overline{f_{-}(z) dz}. 
\]
\end{theorem}

\begin{proof}
We prove first a more abstract statement which seems a basic statement in the theory of Riemann surfaces. However we did not know any reference for this.

\begin{prop}
\label{quadrangulation}
Let $X$  be a connected, compact, non-empty Riemann surface endowed with a quadrangulation. In such a situation, the set of vertices of the quadrangulation is partitioned in two subsets $E^{+}$ and $E^{-}$ such that no two  vertices of $E^{+}$ (resp. $E^{-}$) are connected by an edge. Let $Q$ be the set of faces. For every $q\in Q$, we fix an oriented path $\delta_{q}^{+}$ (resp.  $\delta_{q}^{-}$) whose extremities are the vertices of $q$ in $E^{+}$ (resp. in $E^{-}$). We impose furthermore that the intersection product is given by $\delta_q^{+}\bullet \delta_q^{-}=1$ (here, our convention is that $\delta_q^+$ and $\delta_q^-$ cross counterclockwise).
Let $\omega_{+}$ and $\omega_{-}$ be meromorphic differential forms on $X$ whose poles are at most simple and reside in $E^{+}$ and $E^{-}$ respectively.
Then one has 
$$
\int_{X}\omega_{+}\wedge\bar\omega_{-}=\sum_{q\in Q}\int_{\delta_{q}^{-}}\omega_{+}\int_{\delta_{q}^+}\bar\omega_{-},
$$
where $q$ runs through the faces of the quadrangulation.

\end{prop}
\begin{proof}
Note that the formula makes sense. 
Set $V=E^-$. 
Let $Q$ be the set of faces of the quadrangulation.
We can suppose that for every $q\in Q$, the paths $\delta_{q}^{+}$ and $\delta_{q}^{-}$ are supported on a set that divides $q$ in two connected components. 
Denote by $V_{q}$ the subset of vertices in $V$ that are adjacent to $q$.
For $v\in V$, denote by $\tau(q,v)$ the connected component containing $v$ of $q$ deprived of the support of $\delta_{q}^{+}$.
Denote by $Q_{v}$ the set of faces of the quadrangulation which are adjacent to $v$.
One has a disjoint union (up to a negligible subset)
$$
X=\cup_{v\in V}\cup_{q\in Q_{v}}\tau(q,v).
$$
We turn now to the computation. One has
$$
\int_{X}\omega_{+}\wedge\bar\omega_{-}=\sum_{v\in V}\sum_{q\in Q_{v}}\int_{\tau(v,q)}\omega_{+}\wedge\bar\omega_{-}=\sum_{v\in V}\int_{\cup_{q\in Q_v}\tau(v,q)}\omega_{+}\wedge\bar\omega_{-}.
$$
The boundary of $\cup_{q\in Q_{v}}\tau(v,q)$ is $\sum_{q\in Q_{v}}\alpha(q,v)\delta_{q}^{+}$, where $\alpha(q,v)=1$ (resp. $-1$) if $\delta_q^+$ goes left (resp. right) from $v$'s point of view.

For $v\in V$, and $z\in \cup_{q\in Q_v}\tau(v,q)$ denote by $F_{v}(z)=\int_{v}^{z}\omega_{+}$, which is well defined since $ \cup_{q\in Q_v}\tau(v,q)$ is simply connected. One has $dF_{v}\omega_{-}=\omega_{+}\wedge\omega_{-}$.
By Stokes' theorem, one has
$$
\int_{X}\omega_{+}\wedge\bar\omega_{-}=\sum_{v\in V}\sum_{q\in Q_{v}}\alpha(q,v)\int_{\delta_{q}^{+}}F_{v}(z)\omega_{-}.
$$
Note that $\sum_{q\in Q_{v}}\alpha(q,v)F_v=\int_{\delta_{q}^{-}}\omega_{+}$, because of the condition $\delta_q^{+}\bullet \delta_q^{-}=1$. 
Thus we get the desired formula.
\end{proof}
The theorem follows. Indeed, the translates by $\Ga$ of the fundamental domain $D_0$ of $X(2)$ given by the hyperbolic quadrangle with vertices $0$, $1$, $-1$ and $\infty$ provide a quadrangulation of $X_\Gamma$.
The faces are given by the $\Gamma gD_0$, for $g\in \Gamma\backslash\Gamma(2)$. The $\delta_{\Gamma gD_0}^{+}$  and $\delta_{\Gamma gD_0}^{-}$ can be chosen as the image of the geodesic paths from $g0$ to $g\infty$ and $g1$ to $g(-1)$ respectively.
The Petersson inner product \cite[p. 182]{MR2112196}:
$$
 <f_{+}, f_->
= \frac{1}{ [\SL_2(\Z):\Ga]} \int_{D_{\Ga}} f_{+}(z) \overline{f_-}(z) dx dy
= \frac{1}{12 i[\Ga(2):\Ga]} \int_{X_{\Ga}} \omega_{f_{+}} \wedge \overline{ \omega_{f_{-}}},
$$
where $D_\Ga$ is a fundamental domain for $\Ga$ in the upper half-plane.
The last term can be expressed using proposition \ref{quadrangulation}.
\end{proof}

Let $D^+$ and $D^-$ be divisors of degree $0$ supported on $\partial_\Ga^+$ and $\partial_\Ga^-$ respectively. 

\begin{prop}
\label{orthogonality}
With the notations of the preceeding theorem, suppose furthermore that $f_{+}(z)\, dz$ (resp. $f_{-}(z)\,dz$) is the pullback of a canonical differential form of the third kind. One has
\[
<f_{+}, f_{-}>=0.
\]
\end{prop}
\begin{proof}
Call $\omega_{+}$ the differential form whose pullback is $f_{+}(z)\, dz$. The proposition is true whenever $\omega_+$ is holomorphic on $X_\Ga$. Thus we can suppose that $\omega_+$ is a canonical differential form of the third kind.
It can be written as $\partial h$, where $h$ is a harmonic function on $X_\Ga$. Hence 
$\int_{X_\Ga}\partial h\wedge{\bar\omega}=0$, for all meromorphic differential forms $\omega$ whose poles are concentrated in $\partial_\Ga^-$.
\end{proof}

\section{Periods of Eisenstein series}
\label{EC}
We use the setup of section \ref{Riemann}.
Let $m=2\pi|\Ga/\Ga(2)|$ be the hyperbolic volume of the modular curve $X_\Ga$ (it will only play a transitory role below). Set $q=e^{2i\pi z}$. Consider the function on the upper half-plane:
\[
\log(\eta_{\Ga,j}(z))=-\frac{mz}{4 i}-\pi m[ \sum_{r=1}^{\infty} \phi_{jj r}(1) q^r].
\]
It is connected to the Eisenstein series $E_{j}(z,s)$ via a version of the Kronecker limit theorem established by Goldstein \cite[Theorem 3-1; 3-3]{MR0318065}
\begin{equation}
\label{klf}
\lim_{s \rightarrow 1} [ \frac{1}{2 \pi} E_j(z,s) -\frac{1}{2 \pi m(s-1)}]=\frac{\beta_j}{2}- \frac{1}{\pi m} \log(2) -\frac{1}{\pi m} \log|\sqrt{y} \eta_{\Ga,j}(z)^2|,
\end{equation}
where $z=x+iy$ and $\beta_{j}$ is a complex number independent of $z$ (the scattering constants $C_{j,j}$ up to a scalar multiple independent of $j$).
Goldstein has explored the properties of the function $\eta_{\Ga,j}$, which is a modular form of weight $1/2$ for the group $\sigma_j\Ga\sigma_j^{-1}$, and whose periods he sees as analogues of Dedekind sums \cite{MR0318065}.

The multivalued function on the upper half-plane associated to the divisor $D$ is then given by
\[
\log(\eta_{\Ga,D}(z))=\sum_{j}m_{j}\log(\eta_{\Ga,j})(z).
\]
Note that 
\[
E_D(z):=\frac{2}{m} \log|\eta_{\Ga,D}(z)^2|.
\]
\begin{prop}
\label{exact}
 We have an equality of differential forms
\[
2\pi iG_D(z)\,dz=\frac{4\pi}{m}\,d\eta_{\Ga,D}.
\]
\end{prop}
\begin{proof}
We compare the $q$-expansions. One has, taking into account that $D$ is of degree $0$,
\[
d \eta_{\Ga,D}(z)=2\pi i\sum_{j}m_{j}\, d\log(\eta_{\Ga,j})(z)=2\pi i\sum_{j}-m_j\pi m[ \sum_{r=1}^{\infty} r\phi_{jj r}(1) q^r\,dz].
\]
Similarly, constant terms disappear in $G_D$. So we just need to check that the $q$-expansions corresponding to each cusp agree.
When $j=\Ga\infty$, it is obvious by examining the $q$-expansions.

The other terms agree as well. Indeed, one has $\log(\eta_{\Ga,j})(\sigma_jz)=\log(\eta_{\sigma_j^{-1}\Ga\sigma_j,\Ga\infty(z)})$. 
On the other side we use the identity $G_{j|\sigma_j}=G_{\sigma_j^{-1}\Ga\sigma_j \infty}$, and we are left with the case of the cusp $\sigma_j^{-1}\Ga\sigma_j\infty$.

\end{proof}

Proposition \ref{exact} can be reformulated as 
 \begin{eqnarray*}
 \label{primitive}
 G_{D}(z)\,dz=\frac{2}{mi}\,d \eta_{\Ga,D}(z)=2\pi i\sum_{j\in\partial_\Ga}m_j\, d(\sum_{r=1}^{\infty}\phi_{jj,r}(1)q^{r}).
 \end{eqnarray*}

Recall that
 \begin{eqnarray*}
\gS_D(x)=\frac{1}{2\pi i}\sum_{j\in\partial_{\Ga}}m_{j}\sum_{r=1}^{\infty}\phi_{jj,r}(1)e^{2ri\pi x}.
 \end{eqnarray*}

Equipped with these formulas, we relate the periods of Eisenstein series to the functions defined in the introduction.
\begin{prop}
\label{periods}
One has, for $g\in\Ga(2)$,
\[
\int_{g(1)}^{g(-1)}G_{D^{+}}(z)\,dz=F_{D^{+}}(g)\hskip1truecm\text{and}\hskip1truecm
\int_{g(0)}^{g(\infty)}G_{D^{-}}(z)\,dz=F_{D^{-}}(g).
\]

\end{prop}
\begin{proof}
We apply Proposition ~\ref{exact} and its reformulation Proposition~\ref{primitive} to $D=D^{+}$. Thus we have

\begin{eqnarray*}
\int_{g(1)}^{g(-1)}G_{D^{+}}(z)\,dz
=2\pi i\int_{g(1)}^{g(-1)}\sum_{j\in\partial_\Ga}m_j\sum_{r=1}^{\infty}\phi_{jj,r}(1)q^{r}\,dz
=\gS_{D^+}(g(-1))-\gS_{D^+}(g(1)).
\end{eqnarray*}
The other statement is proved similarly.
\end{proof}

The integrals occuring in Proposition \ref{periods} play a key role in our work. It is tempting to ask whether they have a $p$-adic counterpart, using for instance Coleman's integration. A priori, the notion of canonical differential of the third kind does not make sense without making choice in the $p$-adic world. However, Pierre Colmez suggests that it still does using the notion of Wintenberger splitting \cite{MR1645429}.

\section{Scattering constants}
\label{scattering}
We provide another expression for the functions ${\sI}_{D^{+}}$ and ${\sI}_{D^{-}}$ of the introduction, in terms of scattering constants.
Recall that, for $j$, $k\in \partial_\Ga$, one has
\[
C_{j,k}=\lim_{s\to 1}\pi(\sum_{\sigma \in \Ga_j \backslash \Ga \slash \Ga_k} \frac{1}{|c|^{2s}})-\frac{2}{\pi|\Ga(2)/\Ga|(s-1)}
\]
where the entries $c$ and $d$ are such that  $\sigma_j^{-1} \sigma \sigma_k=\left(\begin{smallmatrix}
\star & \star\\
c & d\\
\end{smallmatrix}\right)$ and $c$ is required to be $>0$.
\begin{prop}
\label{scatteringperiods}
Let $g\in \Gamma(2)$. One has
\[
{\sI}_{D^{+}}=\pi \sum_{g\in \Ga\backslash\Ga(2)}\sum_{j\in \partial_{\Ga}}m_j(C_{j,\Ga g(-1)}-C_{j,\Ga g(1)})\{g0,g\infty\}^+
\]
and 
\[
{\sI}_{D^{-}}=\pi \sum_{g\in \Ga\backslash\Ga(2)}\sum_{j\in \partial_{\Ga}}m_j(C_{j,\Ga g(\infty)}-C_{j,\Ga g(0)})\{g(1),g(-1)\}^-.
\]
\end{prop}
\begin{proof}
We prove the first equality. Let $g\in \Ga$.
We use the formula $\int_{g(1)}^{g(-1)}G_{D^{+}}(z)\,dz=F_{D^{+}}(g)$. Thus we get
\[
F_{D^{+}}(g)=\frac{2}{mi}[\log(\eta_{\Ga,D})]_{g(1)}^{g(-1)}.
\]
Recall that $\sI_{D^{+}}$ is defined as the real part of $2\pi i\sE_{D^+}=2\pi i\sum_{g \in \Ga \backslash \Ga(2)} F_{D^{+}}(g) \{g0,g\infty\}^{+}$. We get
\[
\Re(2\pi iF_{D^{+}}(g))=\Re(\frac{4\pi}{m}[\log(\eta_{\Ga,D})]_{g(1)}^{g(-1)})=\pi[E_D]_{g(1)}^{g(-1)}.
\]
We use the formulas \ref{development} and \ref{constantterm} for the development \ref{development} of $E_{j}(z,s)$ at the cusps $g(1)$ and $g(-1)$.
Note that , when $m_j\ne0$, the cusp $j$ is distinct from both cusps $\Ga g(1)$ and $\Ga g(-1)$, since $\Ga(2)j=\Ga(2)0$ or $\Ga(2)\infty$.
The constant terms given by equation \ref{constantterm} are $\phi_{j\Ga g(1),0}(1)\pi^{1/2}\Gamma(1/2)/\Gamma(1)=\pi\phi_{j\Ga g(1),0}(1)$ and $\pi\phi_{j\Ga g(-1),0}(1)$ respectively. We get
\[
\Re(2\pi if_{D^{+}}(g))=\sum_{j\in\partial_\Ga}\pi^2 m_j(\phi_{j\Ga g(-1),0}(1)-\phi_{j\Ga g(1),0}(1))=\pi\sum_{j\in\partial_\Ga}m_j(C_{j,\Ga g(-1)}-C_{j,\Ga g(1)})
\]
(because of the cancellation of the poles, the middle term makes sense).
The formula follows.
The second formula is proved similarly.

\end{proof}

It is interesting to compare to the work of K\"uhn \cite{MR2140212}. 
For instance, when $D^+$ (resp. $D^-$) is torsion, and the Belyi map is defined over $\Q$, K\"uhn shows that 
\[
e^{2\pi \sum_{j\in \partial_{\Ga}}m_j(C_{j,\Ga g(-1)}-C_{j,\Ga g(1)})}
\]
(resp.  $e^{2\pi \sum_{j\in \partial_{\Ga}}m_j(C_{j,\Ga g\infty}-C_{j,\Ga g0})}$) is a rational number for every $g\in\Ga$.

More precisely, still in the case where the Belyi map is defined over $\Q$, K\"uhn proves that the N\'eron-Tate pairing of two divisors $D=\sum_jm_j[j]$ and $D'=\sum_{j'}m'_{j'}$, both of degree $0$, satisfies 

\begin{equation}
\label{Kuhn}
[D,D']_{\rm NT}\in {\rm log}(\Q^\times_+)+2\pi \sum_{j,j'\in\partial_\Ga}m_jm'_{j'}C_{j,j'}.
\end{equation} 

A similar formula holds even if the Belyi map is not defined over $\Q$; but one needs to consider the conjugates of the Belyi map. The second term in fomula \ref{Kuhn} can be expressed as follows, when $D=D^+$ and $D'=D^-$,
\[
2\pi \sum_{j,j'\in\partial_\Ga}m_jm'_{j'}C_{j,j'}=-2\sI_{D^+}\bullet\sR_{D^-}.
\]
Indeed, we just have to note that the boundary of $\sR_{D^-}$ is $-D^-$ (that will be proved in section \ref{proof}). 
Note that N\'eron-Tate pairing considered by K\"uhn is obtained by summing contribution coming from all places of $\Q$ and is relative to the jacobian $J_\Ga$, whereas our intersection product, of complex analytic nature, takes place relatively to the 1-motives we are considering. One wonders whether an adjustment of the N\'eron-Tate pairing for the mixed situation would not provide a clearer connection between K\"uhn's formula and ours.

\section{Proof of theorem \ref{Eisensteincycles}}
\label{proof}
Let $D$ a divisor supported on $\partial_\Ga^-$ (resp. $\partial_\Ga^+$). Let $c_D$ be a cycle on $X_\Gamma-\partial_\Ga^+$ (resp. $X_\Gamma-\partial_\Ga^-$) of boundary $D$. Then the real number
$\Im(\frac{1}{2i\pi}\int_{c_D}\omega_{D^+})$ (resp. $\Im(\frac{1}{2i\pi}\int_{c_D}\omega_{D^-})$) depends only on $D$. Let $j$, $j_0$ in $\partial_\Ga^-$ (resp  $\partial_\Ga^+$ ). When $D=[j]-[j_0]$, we abuse notations and write
$\Im(\frac{1}{2i\pi}\int_{j_0}^j\omega_{D^-})$ (resp. $\Im(\frac{1}{2i\pi}\int_{j_0}^j\omega_{D^+})$) instead. 

Since $\sum_{j\in \partial_\Ga^-}\lambda^+(j)=0$ (resp. $\sum_{j\in \partial_\Ga^+}\lambda^-(j)=0$), the class $\sum_{j\in\partial_\Ga^-}\Im(\frac{1}{2i\pi}\int_{j_0}^j\omega_{D^+})\lambda^+(j)$ in the real vector space $\HH_1(X_\Gamma-\partial_{\Ga}^{+},\partial_{\Ga}^{-};\R)$ (resp. 
$\sum_{j\in\partial_\Ga^+}\Im(\frac{1}{2i\pi}\int_{j_0}^j\omega_{D^-})\lambda^+(j)$ in $ \HH_1(X_\Gamma-\partial_{\Ga}^{-},\partial_{\Ga}^{+};\R)$)
 is independent of the choice of $j_0$. In other words, $\sum_{j\in\partial_\Ga^-}\Re(\int_{j_0}^j\omega_{D^+})\lambda^+(j)$ (resp. $\sum_{j\in\partial_\Ga^+}\Re(\int_{j_0}^j\omega_{D^-})\lambda^-(j)$) makes sense and is independent of $j_0$.

Recall that $\sI_{D^{+}}$ (resp. $\sI_{D^{-}}$) is defined as the real part of $2\pi i\sE_{D^+}$ (resp. $2\pi i\sE_{D^-}$).

\begin{prop}
\label{imaginary}
One has 
\[
\sI_{D^{+}}=\sum_{j\in\partial_\Ga^-}\Re(\int_{j_0}^j\omega_{D^+})\lambda^+(j)\hskip1truecm {\rm and}\hskip1truecm
\sI_{D^{-}}=-\sum_{j\in\partial_\Ga^+}\Re(\int_{j_0}^j\omega_{D^-})\lambda^-(j).
\]
In particular, those classes belong to the real vector spaces spanned by the images of $\lambda^+$ and $\lambda^-$ respectively. 
\end{prop}
\begin{proof}
By proposition \ref{periods}, one has, by using proposition \ref{boundary}
\begin{eqnarray*}
\sI_{D^+}&=&\sum_{g\in\Ga\backslash\Ga(2)}\Re(\int_{g(1)}^{g(-1)}\omega_{D^+})\xi^+(g)\\
&=&\sum_{g\in\Ga\backslash\Ga(2)}\Re(\int_{j_0}^{g(-1)}\omega_{D^+}-\int_{j_0}^{g(1)})\xi^+(g)\\
&=&\sum_{g\in\Ga\backslash\Ga(2)}\Re(\int_{j_0}^{g(1)}\omega_{D^+})(\xi^+(gB)-\xi^+(g))\\
&=&\sum_{j\in\partial_\Ga^-}\sum_{g\in\Ga\backslash\Ga(2),\Ga g(1)=j}\Re(\int_{j_0}^{g(1)}\omega_{D^+})(\xi^+(gB)-\xi^+(g))\\
&=&\sum_{j\in\partial_\Ga^-}\Re(\int_{j_0}^j\omega_{D^+})\lambda^+(j).
\end{eqnarray*} 
Similarly, one has 
\begin{eqnarray*}
\sI_{D^-}&=&\sum_{g\in\Ga\backslash\Ga(2)}\Re(\int_{g0}^{g\infty}\omega_{D^-})\xi^-(g)\\
&=&\sum_{g\in\Ga\backslash\Ga(2)}\Re(\int_{j_0}^{g\infty}\omega_{D^-}-\int_{j_0}^{g0}\omega_{D^-})\xi^-(g)\\
&=&\sum_{j\in\partial_\Ga^+}\sum_{g\in\Ga\backslash\Ga(2),\Ga g\infty=j)\Re(\int_{j_0}^{g\infty}\omega_{D^-})\xi^-(g)-\sum_{g\in\Ga\backslash\Ga(2),\Ga g0=j}\Re}\int_{j_0}^{g0}\omega_{D^-})\xi^-(g)\\
&=&-\sum_{j\in\partial_\Ga^-, \Ga(2)j=\Ga(2)\infty}\Re(\int_{j_0}^j\omega_{D^+})\lambda^-(j)-\sum_{j\in\partial_\Ga^-, \Ga(2)j=\Ga(2)0}\Re(\int_{j_0}^j\omega_{D^+})\lambda^-(j)\\
&=&-\sum_{j\in\partial_\Ga^+}\Re(\int_{j_0}^j\omega_{D^-})\lambda^-(j).
\end{eqnarray*} 

\end{proof}

\begin{cor}
One has $\sI_{D^{+}}\bullet\sI_{D^{-}}=0$.
\end{cor}
\begin{proof}
Indeed, one has $\lambda^+(j)\bullet \lambda^-(j')$, for $j\in\partial_\Ga^-=0$, $j'\in\partial_\Ga^+$  (small loops do not intersect).
\end{proof}
The following consequence is certainly well-known. However it follows nicely from our proposition.
\begin{cor}
One has $\sI_{D^{+}}=0$ (resp. $\sI_{D^{-}}=0$) if and only if, for all $j$, $j'\in\partial_\Ga^-$ (resp. $j$, $j'\in\partial_\Ga^-$) one has $\Re(\int_{j'}^j\omega_{D^+})=0$ (resp. $\Re(\int_{j'}^j\omega_{D^-})=0$).
\end{cor}
\begin{proof}
Indeed, the kernel of the map $\lambda^+$ : $\Z[\partial_\Ga^-]\rightarrow \HH_1(X_\Gamma-\partial_{\Ga}^{-},\partial_{\Ga}^{+};\Z)$ is spanned by $\sum_{j\in\partial_\Ga^-}[j]$. It follows that all the coefficients in the expression of $\sI_{D^{+}}$ need to be equal for $\sI_{D^{+}}$ to vanish. The other statement is proved similarly.
\end{proof}

\begin{proof}
We prove theorem \ref{Eisensteincycles}.

By proposition \ref{orthogonality} and theorem \ref{innerprod}, one has $\int_{\sE_{D^+}}\omega=0$. We write $\omega=\omega_{D^-}+\omega_0$, where $\omega_0$ is holomorphic on $X_\Ga$ and $\omega_{D^-}$ is the canonical differential form of the third kind associated to $D^-$, where $D^-$ is the divisor of $\omega$. 
Since $\omega_0$ has no pole on $X_\Ga$, proposition \ref{imaginary} implies that $\int_{\sI_{D^{+}}}\omega_0=0$, and therefore $\int_{\sR_{D^{+}}}\omega_0=0$.
Thus one gets $\int_{\sE_{D^+}}\omega=\int_{\sE_{D^+}}\omega_{D^-}=0$. 

One has in $\C$, with the notations of the introduction, and using the formula for the intersection products and proposition \ref{orthogonality},
\[
\sE_{D^+}\bullet\sE_{D^-}=0.
\]
Since $\sI_{D^{+}}\bullet\sI_{D^{-}}=0$, if follows that we have the following identities in $\R$
\[
\sR_{D^+}\bullet\sR_{D^-}=0,
\]
and 
\[
\sR_{D^+}\bullet\sI_{D^-}+\sI_{D^+}\bullet\sR_{D^-}=0.
\]
Thus we get
\[
\Re(\int_{\sR_{D^+}}\omega)=\Re(\int_{\sR_{D^+}}\omega_{D^-})=\Re(\sR_{D^+}\bullet\sR_{D^-}+\frac{1}{2\pi i}\sR_{D^+}\bullet\sI_{D^-})=\sR_{D^+}\bullet\sR_{D^-}=0.
\]

It remains to prove that the boundary of $\sE_{D^+}$ (resp. $\sE_{D^-}$) is ${-D^+}$ (resp. ${-D^-}$). Indeed,
\begin{eqnarray*}
\delta(\sE_{D^+})&=&\sum_{g\in\Ga\backslash\Ga(2)}\int_{g(1)}^{g(-1)}G_{D^{+}}(z)\,dz((\Ga g\infty)-(\Ga g0))\\
&=& \sum_{j\in\partial_{\Ga}^{+}}\sum_{g\in\Ga\backslash\Ga(2), \Ga g\infty=j}\int_{g(1)}^{g(-1)}G_{D^{+}}(z)\,dz(j)
- \sum_{j\in\partial_{\Ga}^{+}}\sum_{g\in\Ga\backslash\Ga(2), \Ga g0=j}\int_{g(1)}^{g(-1)}G_{D^{+}}(z)\,dz(j)\\
&=& \sum_{j\in\partial_{\Ga}^{+}}\sum_{g\in\Ga\backslash\Ga(2), \Ga g\infty=j}\frac{1}{2\pi i}\int_{\xi^-(g)}\omega_{D^+}(j)
-\sum_{j\in\partial_{\Ga}^{+}}\sum_{g\in\Ga\backslash\Ga(2), \Ga g0=j}\frac{1}{2\pi i}\int_{\xi^-(g)}\omega_{D^+}(j).\\
\end{eqnarray*}
Since, by proposition \ref{boundary}, for every $j\in\partial_\Ga^+$, one has $\frac{1}{2\pi i}\int_{\lambda^+(j)}\omega_{D^{+}}=-\sum_{g\in\Ga\backslash\Ga(2), \Ga g\infty=j}\frac{1}{2\pi i}\int_{\xi^-(g)}\omega_{D^+}$ and, similarly,
$\frac{1}{2\pi i}\int_{\lambda^+(j)}\omega_{D^{+}}=\sum_{j\in\partial_{\Ga}^{+}}\sum_{g\in\Ga\backslash\Ga(2), \Ga g0=j}\frac{1}{2\pi i}\int_{\xi^-(g)}\omega_{D^+}$, one gets 
\[
\delta(\sE_{D^+})=-\sum_{j\in\partial_{\Ga}^{+}}{\rm Res}_j(\omega_{D^+})=-D^+.
\]
The boundary of $\sE_{D^-}$ is computed similarly.
\end{proof}

\section{Manin-Drinfeld properties and generalized  Jacobians}
Let $D$ be an effective divisor on $X_\Ga$ supported on $S$. Recall  that the generalized  Jacobian of $X_\Ga$ with respect to $D$ admits a complex analytic description as follows. 
Denote by $\Omega(-D)$ the space of meromorphic differential forms on $X_\Ga$ whose divisor is $\ge -D$. Denote by $E={\rm Hom}_\C(\Omega(-D),\C)$.
Denote by $L$ the image of the map $\HH_1(X_\Ga-S;\Z)\rightarrow E$ which associates to $c$ the map $\omega\mapsto \int_c\omega$.
The complex Lie group formed by the complex points of the generalized  Jacobian of $X_\Ga$ with respect to $D$ is isomorphic to $E/L$. 

Suppose $D=\sum_{s\in S}[s]$, then the group of complex point of the generalized  Jacobian can be identified to a subgroup of $\HH_1(X_\Ga-S;\C)/\HH_1(X_\Ga-S;\Z)$. 
This subgroup is $(\HH_1(X_\Ga-S;\R)+i\lambda(\HH_0(S;\R)))/\HH_1(X_\Ga-S;\Z)$, where $\lambda$ is the canonical map $\HH_0(S;\R)\rightarrow \HH_1(X_\Ga-S;\R)$.

When $D'$ is a divisor of degree $0$ of $X_\Ga$ supported outside $D$, its image in the generalized  Jacobian is the class of the map $\omega\mapsto\int_{c_D'}\omega$, where $c_{D'}\in \HH_1(X_\Ga-S,D';\Z)$ has boundary $D'$.

In the case of interest to us $S=\partial_\Ga^-$ (resp. $\partial^+$) and $D'=D^+$ (resp. $D^-$) is supported on $\partial_\Ga^+$ (resp. $\partial_\Ga^-$).

We define the {\it cuspidal subgroup} of $J_\Ga^-$ (resp. $J_\Ga^+$) as the subgroup spanned by the divisors of degree $0$ supported on $\partial_\Ga^+$ (resp. $\partial_\Ga^-$). To simplify notation, we see $\HH_1(X_{\Ga}-\partial_{\Ga}^{-};\Z)$ 
(resp. $\HH_1(X_{\Ga}-\partial_{\Ga}^{+}; \Z)$) as a subgroup of $\HH_1(X_{\Ga}-\partial_{\Ga}^{-}, \partial_{\Ga}^{+}; \Z)$ (resp.  $\HH_1(X_{\Ga}-\partial_{\Ga}^{+}, \partial_{\Ga}^{-};\Z)$). 

\begin{prop}
\label{generalized Jacobian}
The converse of the  map $c\mapsto(\omega\mapsto \int_c\omega)$ provides the following group isomorphisms : 
\[
J_\Ga^-(\C)\simeq (\HH_1(X_{\Ga}-\partial_{\Ga}^{-};\R)+i\lambda^+(\HH_0(\partial_\Ga^-;\R)))/\HH_1(X_{\Ga}-\partial_{\Ga}^{-}; \Z)
\]
and 
\[
J_\Ga^+(\C)\simeq(\HH_1(X_{\Ga}-\partial_{\Ga}^{+}; \R)+i\lambda^-(\HH_0(\partial_\Ga^+;\R)))/\HH_1(X_{\Ga}-\partial_{\Ga}^{+};  \Z). 
\]
The groups on the right-hand sides identify canonically to subgroups of 
\[
(\HH_1(X_{\Ga}-\partial_{\Ga}^{-}, \partial_{\Ga}^{+}; \R)+i\lambda^+(\HH_0(\partial_\Ga^-;\R)))/\HH_1(X_{\Ga}-\partial_{\Ga}^{-}, \partial_{\Ga}^{+}; \Z)
\]
and 
\[
(\HH_1(X_{\Ga}-\partial_{\Ga}^{+}, \partial_{\Ga}^{-};\R)+i\lambda^-(\HH_0(\partial_\Ga^+,\R)))/\HH_1(X_{\Ga}-\partial_{\Ga}^{+}, \partial_{\Ga}^{-}; \Z)
\] 
respectively. 
In the latter groups, the class of $D^+$ (resp. $D^-$) in $J_\Ga^-(\C)$ (resp. $J_\Ga^+(\C)$) is the class of $-\sE_{D^+}$ (resp. $-\sE_{D^-}$).
\end{prop}
\begin{proof}
The first statement follows from what we have just recalled on generalized  Jacobians. We prove now the second statement regarding $D^+$. Let $c\in \HH_1(X_{\Ga}-\partial_{\Ga}^{-}, \partial_{\Ga}^{+};\Z)$ such that the boundary of $c$ is $D^+$. 
By Abel's theorem, the image of $D^+$ in $J_\Ga^-(\C)$ is the map $\omega\mapsto\int_c\omega$. Let $\omega$ be a meromorphic differential form of the third kind with divisor supported on $\partial_\Ga^-$. 
One has $\int_{\sE_{D^+}}\omega=0$. Thus $\int_c\omega=\int_{c-\sE_{D^+}}\omega$. The boundary of $c-\sE_{D^+}\in \HH_1(X_{\Ga}-\partial_{\Ga}^{+}, \partial_{\Ga}^{-}; \R)+i\lambda^-(\HH_0(\partial_\Ga^+;\R))$ is $0$. The desired result follows. 
The assertion concerning $D^-$ is proved similarly. 

\end{proof}

\begin{cor}
\label{cuspidalgroup}
The cuspidal subgroup of the generalized  Jacobian $J_\Ga^-$ (resp. $J_\Ga^+$) is isomorphic to the subgroup of $\HH_1(X_{\Ga}-\partial_{\Ga}^{-}, \partial_{\Ga}^{+}; \C)/\HH_1(X_{\Ga}-\partial_{\Ga}^{-}, \partial_{\Ga}^{+}; \Z)$ 
(resp. $\HH_1(X_{\Ga}-\partial_{\Ga}^{+}, \partial_{\Ga}^{-}; \C)/\HH_1(X_{\Ga}-\partial_{\Ga}^{+}, \partial_{\Ga}^{-}; \Z)$) generated by the classes $\sE_{D^+}$ (resp. $\sE_{D^-}$) when $D^+$ runs through the divisors of degree $0$ supported on $\partial_\Ga^+$ (resp. $\partial_\Ga^-$).
\end{cor}
We now turn to theorem \ref{Manin-Drinfeld}. The first part follows directly from the following proposition.
\begin{cor}
\label{torsion}
The divisor $D^+$ (resp.  $D^-$) is torsion in $J_\Ga^-$ (resp. $J_\Ga^+$) if and only if one has $\sR_{D^+}\in \HH_1(X_{\Ga}-\partial_{\Ga}^{-}, \partial_{\Ga}^{+}; \Q)$ (resp. $\sR_{D^-}\in \HH_1(X_{\Ga}-\partial_{\Ga}^{+}, \partial_{\Ga}^{-}; \Q)$) and $\sI_{D^+}=0$ (resp. $\sI_{D^-}=0$).
\end{cor}
\begin{proof}
It follows from proposition {generalizedjacobian} that $D^+$ is torsion if and only if $\sE_{D^+}\in \HH_1(X_{\Ga}-\partial_{\Ga}^{-}, \partial_{\Ga}^{+}; \Q)$ which translates immediately into $\sR_{D^+}\in \HH_1(X_{\Ga}-\partial_{\Ga}^{-}, \partial_{\Ga}^{+}; \Q)$ and $\sI_{D^+}=0$. 
The alternate statement is proved similarly.
\end{proof}

The second part of  theorem \ref{Manin-Drinfeld} can be deduced from the first part. Indeed, it follows from the fact that the image of $\lambda^+$ (resp. $\lambda^-$) is the kernel of the map $\HH_1(X_{\Ga}-\partial_{\Ga}^{-}, \partial_{\Ga}^{+}; \Z)\rightarrow \HH_1(X_{\Ga}, \partial_{\Ga}^{+}; \Z)$ 
(resp. $\HH_1(X_{\Ga}-\partial_{\Ga}^{+}, \partial_{\Ga}^{-}; \Z)\rightarrow \HH_1(X_{\Ga}, \partial_{\Ga}^{-}; \Z)$). 

\section{Passage to any modular curve }
\label{GeneralEisensteinclasses}

Our purpose in this section is to derive from Theorem~\ref{Eisensteincycles}  all Eisenstein classes in $\HH_1(X_{\Ga},\partial_{\Ga};\R)$, where $\Ga$ is still a subgroup of $\Ga(2)$. In fact we will need to apply Theorem~\ref{Eisensteincycles} for several conjugates of $\Ga$ in the full modular group.

Recall that Manin \cite{MR0314846} has defined map $\xi$ : $\Z[\Ga\backslash\SL_2(\Z)]\rightarrow \HH_1(X_{\Ga},\partial_{\Ga};\Z)$, which to $[\Ga g]$ associates the class $\{g0,g\infty\}$ of the image in $X_\Ga$ of the path from $g0$ to $g\infty$ in the upper half-plane. Such a map is surjective and its kernel is fully described by Manin.

Denote by $U=\left(\begin{smallmatrix}
0 & 1\\
-1 & 1\\
\end{smallmatrix}\right)$ and $S=\left(\begin{smallmatrix}
0 & 1\\
-1 & 0\\
\end{smallmatrix}\right)$. The element $U$ permutes cyclicly $0$, $1$ and $\infty$. Denote by $\partial^{0}$, $\partial^{1}$ and $\partial^{\infty}$ the cusps of $X_{\Ga}$ above the cusps $\Ga(2)0$, $\Ga(2)1$ and $\Ga(2)\infty$ respectively. Let $D$ be a divisor of degree $0$ supported on $\partial_{\Ga}$. Write 
\[
D=-E_{0}+E_{\infty}
\] 
where $E_{0}$ and $E_{\infty}$ are divisors of degree $0$ supported on $\partial^{\infty}\cup \partial^{1}$ and $\partial^{0}\cup \partial^{1}$ respectively. This decomposition is well-defined up to a divisor of degree $0$ supported on $\partial^{1}$.

When $D$ is supported on $\partial_\Ga^+$ (resp. $\partial_\Ga^0$),
we extend the function $F_D$ to a map still denoted $F_D$ : $\SL_2(\Z)\rightarrow\C$ by $F(g)=0$ whenever $g\notin\Ga$.
\begin{prop}
\label{ExplicitEisenstein}
The element 
\[
\sE'_D=\sum_{g\in \Ga\backslash\SL_2(\Z)}(F_{U^{-1}E_\infty}(U^{-1}g)-F_{UE_0}(Ug))\xi(g)
\]
belongs to $\HH_1(X_{\Ga},\partial(X_{\Ga});\R)$ and is the Eisenstein class of boundary $D$. 
\end{prop}
\begin{proof}
Note first that $UE_{0}$ and $U^{-1}E_{\infty}$ are divisors of degree $0$ supported on $\partial_{U\Ga U^{-1}}^{+}\subset X_{U\Ga U^{-1}}$ and $\partial_{U^{-1}\Ga U}^{+}\subset X_{U^{-1}\Ga U}$ respectively.
Consider the corresponding Eisenstein elements $\sE_{UE_{0}}$ and  $\sE_{U^{-1}E_{\infty}}$ and their respective images $\sE'_{UE_{0}}$ and  $\sE'_{U^{-1}E_{\infty}}$ in $\HH_1(X_{U\Ga U^{-1}},\partial_{U\Ga U^{-1}},\C)$ and $\HH_1(X_{U^{-1}\Ga U},\partial_{U^{-1}\Ga U},\C)$. Those images have in fact real coefficients, see section \ref{Eisensteincycles}.
Then $U^{-1}\sE'_{UE_{0}}$ and $U\sE_{U^{-1}E'_{\infty}}$ are Eisenstein elements in $\HH_1(X_{\Ga},\partial_\Ga;\R)$ of boundary $E_{0}$ and $E_{\infty}$ respectively. 

For $g\in \Ga(2)$, the image of $\xi^+(g)$ via the composition of maps
\[
\HH_1(X_{\Ga}-\partial_{\Ga}^{-}, \partial_{\Ga}^{+}, \Z)\rightarrow \HH_1(X_{\Ga}, \partial_{\Ga}^{+}, \Z)\rightarrow \HH_1(X_{\Ga}, \partial_{\Ga}, \Z)
\]
is $\xi(g)$. Thus, one has
\[
\sE'_{U^{-1}E_{\infty}}=\sum_{g\in U^{-1}\Ga U\backslash\Ga(2)}F_{U^{-1}E_{\infty}}(g)\xi(g)\hskip1truecm\text{and}\hskip1truecm \sE'_{UE_{0}}=\sum_{g\in U\Ga U^{-1}\backslash\Ga(2)}F_{UE_{0}}(g)\xi(g).
\]
Those Eisenstein elements have boundary $U^{-1}E_\infty$ and $UE_0$ respectively. Thus $U\sE_{U^{-1}E_{\infty}}-U^{-1}\sE_{UE_{0}}$ is an Eisenstein elements of boundary $E_\infty-E_0=D$. It is precisely $\sE'_D$, as can be seen by changing variable in the sums. 

\end{proof}

When $\Gamma'$ is a subgroup of ${\rm SL}_2(\Z)$ containing $\Gamma$, consider the degeneracy map $\pi$ : $X_\Ga\rightarrow X_{\Ga'}$. It defines a map $\pi_*$ : $\HH_1(X_{\Ga},\partial(X_{\Ga});\R)\rightarrow \HH_1(X_{\Ga'},\partial(X_{\Ga});\R)$, which is compatible with the boundary maps and respects Eisenstein elements. Thus proposition \ref{ExplicitEisenstein} provides a formula for all Eisenstein elements for modular curves associated to any subgroup of finite index of ${\rm SL}_2(\Z)$.
The formula is unsatisfactory insofar as the decomposition $D=E_\infty-E_0$ is not unique.

Note a dual statement using $\xi^-$ instead of $\xi^+$ is missing. We make no use of $\xi^-$ in proposition \ref{ExplicitEisenstein}.

\begin{cor}
\label{torsioncriterion}
The class of the divisor $D$ is torsion in $J_\Ga$ if and only if there exists $F'$, $F''$ : $\Z[\Ga\backslash\SL_2(\Z)]\rightarrow\R$ which are right $S$-invariant and right $U$-invariant respectively such that for all $g\in {\rm SL}_2(\Z)$
\[
(F_{U^{-1}E_\infty}(U^{-1}g)-F_{UE_0}(Ug))-F'(g)+F''(g)\in \Q.
\]
\end{cor}
\begin{proof}
Indeed, the class of the divisor $D$ is torsion if and only if the Eisenstein element $\sE_D$ belongs to  $\HH_1(X_{\Ga},\partial(X_{\Ga});\Q)$. 
The kernel of Manin's map $\xi$ is spanned by the sum of elements invariant by $U$ and elements invariant by $S$ \cite{MR0314846}. The corollary follows. 

\end{proof}

\section{Fermat curves}
\label{Fermat}

Let $N$ be an integer $>0$. The  $N$-th Fermat curve $F_{N}$ is given by the projective equation:
\[
X^N+Y^N=Z^N.
\]
Fermat curves and their points at infinity (cusps) are studied extensively by Rohrlich \cite{MR2626317}, \cite{MR0441978}, V\'elu \cite{MR582434} and  Posingies~\cite{Posingies}. We use the notations of the latter author.
In particular, we have the Belyi map
\[
\beta_N:F_N \rightarrow \sP^1
\]
given by $(X:Y:Z) \rightarrow (X^N:Z^N)$. The map $\beta_N$ satisfy the following properties:
\begin{itemize}
\item It is of degree $N^{2}$.
\item
It is ramified only above the points $0,1,\infty$. 
\item 
The corresponding ramification points are given by $a_j=(0:\zeta^j:1)$, $b_j=(\zeta^j:0:1)$, $c_j=(\epsilon\zeta^j:1:0)$, for $j\in\Z/N\Z$.
\end{itemize}
Here, $\zeta=e^{2\pi i /N}$is the primitive $N$ th root of unity and $\epsilon=e^{\pi i /N}$. Each of the 
above points has ramification index $N$ over $\sP^1$.

Since the $\lambda$ function identifies $\sP^1-\{0,1,\infty\}$ to $Y(2)$.
A covering of $\sP^1-\{0,1,\infty\}$ can be understood as a covering of $Y(2)$, {\it i.e.} a modular curve.
Consider the group morphism $P\Ga(2)\rightarrow (\Z/N\Z)^{2}$ which sends $A$ to $(1,0)$ and sends $B$ to $(0,1)$. 
An explicit formula for this morphism in terms of the entries of the matrix (and not in terms of its decomposition as product of generators) is due to Murty and Ramakrishnan \cite{MR983619}.
Denote by $\Phi_{N}$ its kernel, which is a subgroup of index $N^{2}$ of $P\Ga(2)$. Denote by $X_{\Phi(N)}$ the corresponding modular curve.

A system of representatives for the cosets $\Phi_N \backslash \Ga(2)$ is given by $A^a B^b$ with $a,b \in \{0,1...(N-1)\}$.  

The dessin d'enfant of the Fermat curve enjoys a special combinatorial property: its edges are in a one-to-one correspondence with the pairs of opposite color vertices. 
\begin{prop}
\label{coordinates}
The map $ \Phi_N \backslash \Ga(2)\rightarrow \Phi_N\backslash\Ga(2)0\times \Phi_N\backslash\Ga(2)\infty$ which to $\Phi_Ng$ associates $(\Phi_Ng0,\Phi_Ng\infty)$ is bijective.
\end{prop}
\begin{proof}
All sets involved have cardinality $N^2$. It is sufficient to show that the map is injective. For the first statement, let $g,g'\in \Ga(2)$ such that $(\Phi_Ng0,\Phi_Ng\infty)=(\Phi_Ng'0,\Phi_Ng'\infty)$. There exists $a$, $b\in\Z$ such that $\Phi_Ng=\Phi_Ng'A^a$ and $\Phi_Ng=\Phi_Ng'B^b$. 
Thus one has $\Phi_NgA^{-a}=\Phi_NgB^{-b}$. Hence $a$ and $b$ belong to $N\Z$. Thus $\Phi_Ng=\Phi_Ng'$.

\end{proof}

Thus the corresponding dessin d'enfant admits as $0$-vertices the cusps $\Phi_{N}A^{a}0$, for $a\in\{0,1,..., N-1\}$, as $\infty$-vertices the cusps $\Phi_{N}B^{b}\infty$,  for $b\in\{0,1,..., N-1\}$, as edges the pairs $\{\Phi_{N}A^{a}0,\Phi_{N}B^{b}\infty\}$ for $a$, $b\in\{0,1,..., N-1\}$. The cyclic ordering ({\it i.e.} action of $\Z$) on the edges attached to $\Phi_{N}A^{a}0$ (resp. $\Phi_{N}B^{b}\infty$) is such that the successor of  $\{\Phi_{N}A^{a}0,\Phi_{N}B^{b}\infty\}$ is $\{\Phi_{N}A^{a+1}0,\Phi_{N}B^{b}\infty\}$ (resp. $\{\Phi_{N}A^{a}0,\Phi_{N}B^{b-1}\infty\}$).

The $N$-th roots
\[
x:=\sqrt[N]{\lambda}, y:=\sqrt[N]{1-\lambda}
\]
define modular units for $\Phi_{N}$. We recover thus the familiar model of the Fermat curve.

The cusps $a_j,b_j$ and $c_j$ have been introduced above.  The divisors of the following modular functions 
are given by:
\[
{\rm div}(x-\zeta^j)=N b_j-\sum_j c_j, 
{\rm div}(y-\zeta^j)=N a_j-\sum_j c_j, 
{\rm div}(x-\epsilon\xi^j y)=N c_j-\sum_j c_j. 
\]

It follows that the cuspidal subgroup of  $J_{\Phi_N}$ is annihilated by $N$. 
Suppose from now that $N$ is an odd integer. 
Rohrlich  \cite{MR0441978} has determined the structure of the cuspidal group of $J_{\Phi_N}$ (we use V\'elu's alternative proof and description \cite{MR582434} ). 
\begin{thm}{(Rohrlich, ~\cite{MR0441978}, \cite{MR582434})}
\label{Rohrlich}
The group of principal divisors of $X_{\Phi(N)}$ is spanned by  by $N[\partial_{\Phi_{N}}]^{0}$ together with the following set (the {\rm Rohrlich relations})
$$
\{\sum_{i=0}^{N-1}[a_{i}]-[P],\sum_{i=0}^{N-1}[b_{i}]-[P],\sum_{i=0}^{N-1}[c_{i}]-[P],
$$
$$
\sum_{i=0}^{N-1}i([a_{i}]-[b_{i}]),\sum_{i=0}^{N-1}i([a_{i}]-[c_{i}]),\sum_{i=0}^{N-1}i^{2}([a_{i}]+[b_{i}]+[c_{i}]-3[P])\},
$$
where $P$ is any cusp of $X_{\Phi_N}$. Consequently the cuspidal subgroup of $J_{\Phi_N}$ is isomorphic to $(\Z/N\Z)^{3N-7}$.
\end{thm}

We turn now to the analogue of this result for the generalized  Jacobian. We start with the Eisenstein cycles. 

\begin{thm}
\label{periodcomputation}
Let $j$, $k\in{\Z}/N{\Z}$,  one has 
  \[
\sE_{(a_j)-(c_k)}= \frac{1}{N}(\sum_{g \in \Phi_N \backslash \Ga(2), g0=a_j}\{g0,g \infty\}^+-\sum_{g \in \Phi_N \backslash \Ga(2), g\infty=c_k}\{g0,g \infty\}^+)
  \]
  and 
 \begin{eqnarray*}
\sE_{(b_j)-(b_k)}&=& \frac{1}{2N}(\sum_{g \in \Phi_N \backslash \Ga(2), g1=b_j}\{g(1),g(-1)\}^--\sum_{g \in \Phi_N \backslash \Ga(2), g(-1)=b_j}\{g(1),g(-1)\}^-)\\
&&- \frac{1}{2N}(\sum_{g \in \Phi_N \backslash \Ga(2), g1=b_k}\{g(1),g(-1)\}^--\sum_{g \in \Phi_N \backslash \Ga(2), g(-1)=b_k}\{g(1),g(-1)\}^-).
  \end{eqnarray*}
  \end{thm}
 \begin{proof}
 Set $f^+_{j,k}=\frac{y-\zeta^j}{x-\epsilon\zeta^ky}$. It is a modular unit of divisor $N (a_j-c_k)$. One has 
 \[
 \omega_{(a_j)-(c_k)}
=d(\log f_{j,k})=\frac{dy}{y-\zeta^j}-\frac{dy}{x-\epsilon\zeta^k y}.
 \]
 Recall that $2\pi i G_{(a_j)-(c_k)}\,dz$ is the pullback on the upper half-plane of $ \omega_{(a_j)-(c_k)}$.
Thus one has
\begin{eqnarray*}
F_{(a_j)-(c_k)}(g)
 &= & \frac{-1}{2 \pi i N}[\int_{g(-1)}^{g(1)}\frac{dy}{y-\zeta^j}-\int_{g(-1)}^{g(1)}\frac{dy}{x-\epsilon\zeta^k y}]\\
 &= & \frac{-1}{2 \pi i N} [\int_{\text{loop around } g(0)}\frac{dy}{y-\zeta^j}-\int_{\text{loop around} g(\infty)}\frac{dy}{x-\epsilon \zeta^k y}]\\
 &= &\frac{-1}{N} [\delta_{g 0, a_j}-\delta_{g \infty, c_k}]\\
  \end{eqnarray*}
  where $\delta$ is the Kronecker symbol.
 The result follows from the formula 
 \[
\sE_{(a_j)-(c_k)}=\sum_{g\in \Phi_N\backslash \Ga}\int_{g(1)}^{g(-1)}G_{(a_j)-(c_k)}(z)\,dz\{g0,g \infty\}^+. 
\]
 
 The second statement is proved similarly, using the function $f^-_{j,k}=(x-\zeta^j)/(x-\zeta^k)$, a modular unite of divisor $N((b_j)-(b_k))$. Thus 
 \begin{eqnarray*}
F_{(b_j)-(b_k)}(g)
 &= & \frac{-1}{2 \pi i N}[\int_{g0}^{g\infty}\frac{dy}{y-\zeta^j}-\int_{g0}^{g\infty}\frac{dy}{x-\epsilon\zeta^k y}]\\
 &= & \frac{-1}{2 \pi i N} [\int_{\text{half-loop around}g(1)}\frac{dy}{x-\zeta^j}-\int_{\text{half-loop around} g(-1)}\frac{dy}{x-\zeta^j}]\\
  & &+ \frac{1}{2 \pi i N} [\int_{\text{half-loop around}g(1)}\frac{dy}{x-\zeta^k}-\int_{\text{half-loop around} g(-1)}\frac{dy}{x-\zeta^k }]\\
 &= &\frac{-1}{2N} [\delta_{g(1), b_j}-\delta_{g(-1), b_k} -\delta_{g(1), b_j}+\delta_{g (-1), b_k}].\\
  \end{eqnarray*}

 \end{proof}
 An evident consequence of theorem is that the imaginary parts of $\sE_{(a_j)-(c_k)}$ and $\sE_{(b_j)-(b_k)}$ vanish. This is in fact a consequence of proposition \ref{scatteringperiods} and of theorem 8.1. of , where Posingies determines the scattering constants for the Fermat curves.
 \begin{cor}
 \label{cuspidalsubgroup}
 The kernel of the map $\Z[\partial_{\Phi_N}^+]^0\rightarrow J_{\Phi_N}^-$ is generated by $N\Z[\partial_{\Phi_N}^+]^0\cup \{\sum_{j=0}^{N-1}[a_{j}]-[c_j]\}$ Therefore the cuspidal subgroup of $J_{\Phi_N}^-$ is isomorphic to $(\Z/N\Z)^{2N-2}$.

The kernel of the map $\Z[\partial_{\Phi_N}^-]^0\rightarrow J_{\Phi_N}^+$ is generated by $2N\Z[\partial_{\Phi_N}^-]^0$. Therefore the cuspidal subgroup of $J_{\Phi_N}^+$ is isomorphic to $(\Z/2N\Z)^{N-1}$. 
\end{cor}
 \begin{proof}
 
We use corollary \ref{cuspidalgroup} to understand the cuspidal subgroup of $J_{\Phi_N}^-$ as a subgroup of $\HH_1(X_{\Phi_N}-\partial_{ \Phi_N}^{-}, \partial_{ \Phi_N}^{+}; \Q)/\HH_1(X_{\Phi_N}-\partial_{ \Phi_N}^{-}, \partial_{ \Phi_N}^{+}; \Z)$. 
Since the denominator of $\sE_{(a_j)-(c_k)}$ is $N$, by theorem \ref{periodcomputation}, the cuspidal subgroup is contained in $\frac{1}{N}\HH_1(X_{\Phi_N}-\partial_{ \Phi_N}^{-}, \partial_{ \Phi_N}^{+};\Z)/\HH_1(X_{\Phi_N}-\partial_{ \Phi_N}^{-}, \partial_{ \Phi_N}^{+}; \Z)$.
Recall that the group $\HH_1(X_{\Phi_N}-\partial_{ \Phi_N}^{-}, \partial_{ \Phi_N}^{+}; \Z)$ is freely generated by the $\{g0,g\infty\}^+$ when $g$ runs through $ \Phi_N\backslash\Ga(2)$ by theorem \ref{int-iso}.
Thus we are left to determine the kernel of the map :
 \[
 \theta^+: (\Z/N\Z)[\partial_{\Phi_N}^+]^0\rightarrow (\Z/N\Z)[ \Phi_N\backslash\Ga(2)]
\]
which to $(a_j)-(c_k)$ associates $\sum_{g \in \Phi_N \backslash \Ga(2), g0=a_j}[g]-\sum_{g \in \Phi_N \backslash \Ga(2), g\infty=c_k}[g]$.
We use now proposition \ref{coordinates}. It follows that the kernel of $\theta^+$ is spanned by $\sum_{j=0}^{N-1}[a_{j}]-[c_j]$. 

The other case can be treated in a similar manner. The group $\HH_1(X_{\Phi_N}-\partial_{ \Phi_N}^{+}, \partial_{ \Phi_N}^{-}; \Z)$ is freely generated by the $\{g0,g\infty\}^-$ when $g$ runs through $ \Phi_N\backslash\Ga(2)$.
Therefore we have to examine the map 
 \[
 \theta^-: (\Z/2N\Z)[\partial_{\Phi_N}^+]^0\rightarrow (\Z/2N\Z)[ \Phi_N\backslash\Ga(2)]
\]
which to $(b_j)-(b_k)$ associates $\sum_{g \in \Phi_N \backslash \Ga(2), g1=b_j}[g]-\sum_{g \in \Phi_N \backslash \Ga(2), g1=b_k}[g]-\sum_{g \in \Phi_N \backslash \Ga(2), g(-1)=b_j}[g]+\sum_{g \in \Phi_N \backslash \Ga(2), g(-1)=b_k}[g]$.

It is induced by the map 
 \[
 \theta': (\Z/2N\Z)[\partial_{\Phi_N}^+]\rightarrow (\Z/2N\Z)[ \Phi_N\backslash\Ga(2)]
\]
which to $(b_j)$ associates $\sum_{g \in \Phi_N \backslash \Ga(2), g1=b_j}[g]-\sum_{g \in \Phi_N \backslash \Ga(2), g(-1)=b_j}[g]$. The kernel of $\theta'$ is spanned by the element $\sum_j(b_j)$. Thus $\theta^-$ is injective.

 \end{proof}

 \begin{remark}
We did not use the sums coming from the Kloosterman zeta function to determine the Eisenstein classes for the Fermat curves. 
A priori, it seems difficult to understand a sum over the elements of $\Phi_N$. Indeed, how to characterize the elements of $\Phi_N$ in terms of the entries of the corresponding matrices? Murty and Ramakrishnan \cite{MR983619} provide an answer in terms of Dedekind sums, but it is unclear how this would be useful.

\end{remark}

\bibliographystyle{crelle}
\bibliography{DEBARGHAMERELEISENSTEIN.bib}
\end{document}